\documentclass{jcmlatex}
\def\JCMvol{xx}
\def\JCMno{x}
\def\JCMyear{200x}
\def\JCMreceived{xxx}
\def\JCMrevised{xxx}
\def\JCMaccepted{xxx}

\usepackage[utf8]{inputenc}
\usepackage[T1]{fontenc}
\usepackage{cite}
\usepackage{amsmath}
\usepackage{amsfonts}
\usepackage{amssymb}
\usepackage{graphicx}
\usepackage{caption}
\usepackage{geometry}
\usepackage{framed}
\usepackage{color}
\usepackage{subfig}
\usepackage[super,square]{natbib}

\begin{document}
	
	\markboth{L.~Feng AND Z.~Huang}{Petrov-Galerkin method for turning point 
	problems}
	
	\title{A Uniform Convergent Petrov-Galerkin method for a Class of Turning 
	Point Problems\footnote{This work was partially supported by NSFC Projects No.~11871298, 
		~12025104.}
}
	
	\author{Li Feng
		\thanks{Department of Mathematical Sciences, Tsinghua University, Beijing, China 
		\\ 
		Email: feng-l18@mails.tsinghua.edu.cn}
		\and
		Zhongyi Huang\footnote{Corresponding author.}
		\thanks{Department of Mathematical Sciences, Tsinghua University, Beijing, China 
		\\ 
		Email: zhongyih@mail.tsinghua.edu.cn}}
	
	\maketitle

	\captionsetup[figure]{labelsep=space}
	\captionsetup[table]{labelsep=space}
	\begin{abstract}
		In this paper, we propose a numerical method for turning point problems in 
		one dimension based on Petrov-Galerkin finite element method (PGFEM).
		We first give a priori estimate for the turning point problem with a single 
		boundary turning point. Then we use PGFEM to solve it,
		where test functions are the solutions to piecewise approximate dual 
		problems.
		We prove that our method has a first-order convergence rate in both 
		$L^\infty$ norm and an energy norm when we select the exact 
		solutions to dual problems as 
		test functions.
		Numerical results show that our scheme is efficient for turning 
		point problems with different types of singularities, and the 
		convergency coincides with our theoretical results.
	\end{abstract}



	\section{Introduction}\label{sec:1}
	Singularly perturbed problems have been widely studied in the fields of 
	fluid mechanics, aerodynamics, convection-diffusion process, etc.
	In such problems, there exist boundary layers or interior layers because a 
	small parameter is included in the coefficient of the highest derivative.
	Consider the following singularly perturbed turning point problem 
	in one dimension:
	\begin{equation}
		\label{eq1}
		\left\{
		\begin{aligned}
			& Lu=-\varepsilon u''+p(x)u'+b(x)u=f(x),\quad x_L<x<x_R,\\
			& u(x_L) = u_L, \quad u(x_R) = u_R,\\
		\end{aligned}\right.
	\end{equation}
	where $p(x)$ has zeros $z_1<z_2<\cdots<z_m$ on $[x_L,x_R]$. We 
	assume $p$, 
	$b$, and $f$ 
	to be sufficiently smooth. Furthermore, we suppose
	\begin{equation}
		b(x)-p'(x)\geq 
		\gamma_0>0
		\label{eq:condition}
	\end{equation}
	to ensure the well-posedness of the dual problems. Each zero 
	of $p(x)$ is presumed to be a single root, i.e., $p'(z_i)\neq0$.
	
	From the asymptotic analysis, we know that there will be boundary/interior 
	layers at some of $z_i$'s.
	Here we consider the following types of singularities:
	\begin{enumerate}
		\item[(a)] Exponential boundary layers (singularly perturbed 
		problems without turning points);
		\item[(b)] Cusp-like interior layers (interior turning point 
		problems);
		\item[(c)] Boundary layers of other types (boundary turning point 
		problems).
	\end{enumerate}
	
	Singularly perturbed elliptic equations without turning 
	points have been widely studied by researchers.
	Various numerical methods are utilized, where finite 
	difference methods and finite element methods play prominent roles. 
	El-Mistikawy and Werle raise an 
	exponential box scheme \cite{el1978numerical}(EMW scheme) in order to solve 
	Falkner-Skan equations.  
	Kellogg, Berger, and Han\cite{berger1984priori}, Riordan and
	Stynes\cite{stynes1986finite}, etc., find this EMW scheme 
	efficient when solving singularly perturbed elliptic equations.
	Fitted operator numerical methods, such as exponentially 
	fitted finite difference method and Petrov-Galerkin method, are developed. 
	Another class of methods, fitted mesh 
	methods\cite{guo1993uniformly, sun1995finite, hemker2000varepsilon, 
		li2000convergence, de2011parameter}, show good adaptivity to 
	different problems, while remeshing is necessary in some moving front 
	problems.
	
	A turning point problem is a class of equations in which the 
	coefficient $p(x)$ vanishes in the domain. Compared to singularly 
	perturbed equations without turning points, interior 
	layers and other types of boundary layers might appear in the solutions to 
	turning point problems. 
	O'Malley\cite{o1970boundary} 
	and Abrahamson\cite{abrahamson1977priori} analyze turning point problems 
	in some common cases.
	Kellogg, Berger, and Han\cite{berger1984priori} theoretically examine 
	turning point problems with single interior turning points, and they use a 
	modified EMW scheme, which obtains a first-order (or lower) convergence 
	rate.
	Stynes and Riordan \cite{stynes1986finite,stynes19871} build a 
	numerical scheme under Petrov-Galerkin framework and prove the uniform 
	convergence 
	in $L^1$ norm and $L^{\infty}$ norm.
	Farrell\cite{farrell1988sufficient} 
	proposes sufficient conditions for an exponentially fitting difference 
	scheme to be uniformly convergent for a turning point problem. 
	Farrell and Gartland\cite{farrell1988uniform} modify the EMW scheme and 
	construct a 
	scheme 
	with uniform first-order convergency, where parabolic cylinder functions 
	are used in the computation. For other studies of turning point problems 
	using fitted operator methods, please refer to \cite{roos1990global, 
		vulanovic1990numerical,geng2014numerical,munyakazi2019robust};
	for fitted 
	mesh methods, please refer to \cite{sun1994finite, 
		natesan2003parameter,chen2008stability,o2011parameter,
		o2012singularly,kumar2019parameter,yadav2021almost}.
	
	We notice that most of the present research 
	assume turning points to be away from the boundary. If a turning point 
	meets an endpoint, the problem is called \textit{a boundary turning point 
		problem}, 
	which has 
	not been thoroughly studied.
	In \cite{vulanovic1987non}, Vulanovi{\'c} considers a turning point problem 
	with an arbitrary single turning point and obtains uniform convergency 
	using finite difference method on a non-equidistant mesh.
	Vulanovi{\'c} and Farrell\cite{vulanovic1993continuous} examine a 
	multiple boundary turning point problem and make priori estimates. 
	However, estimates for 
	single boundary turning point problems and numerical methods based on the 
	uniform mesh are not given yet.
	In order to fill this blank, in this paper we estimate the derivatives of 
	the solution to a standard single boundary turning point problem and raise 
	an algorithm without particular mesh generation.
	
	Petrov-Galerkin finite element method (PGFEM) is used in many 
	problems. Dated 
	back to 1979, Hemker and Groen\cite{de1979error} raise a method that treats 
	problem (1) with Petrov-Galerkin method, where the coefficient $p(x)$ has a 
	positive
	lower bound. The scheme in Farrell and Gartland \cite{farrell1988uniform} 
	is 
	based on the so-called patched function method, also interpreted as 	
	Petrov-Galerkin method. In references 
	\cite{broersen2014robust,chakraborty2020optimal}, Petrov-Galerkin method 
	and discontinuous 
	Petrov-Galerkin method are implemented in elliptic equations in two 
	dimensions, demonstrating their efficiency and convergency.
	
	Tailored Finite Point Method (TFPM) is raised by Han, Huang, and 
	Kellogg\cite{han2008tailored}, which is designed to solve PDEs using 
	properties of the solutions, especially for singularly perturbed 
	problems.
	TFPM could handle exponential singularities well, while simple difference 
	methods might sometimes suffer from a low convergence rate. 
	TFPM is later utilized in interface problems\cite{huang2009tailored},
	steady-state reaction-diffusion equations\cite{han2010tailored}, 
	convection-diffusion-reaction equations\cite{han2013tailored}, etc.
	
	This study presents a numerical scheme to solve problem (1) with 
	several types of singularities. We prove that the width of the boundary 
	layer of a single boundary turning point problem is 
	$O(\sqrt{\varepsilon})$, which is a weaker version of the result in 
	\cite{vulanovic1993continuous}. The derivative of the solution, $u'$, is 
	bounded by $C(1+\varepsilon^{-1/2})$ near the boundary layer and 
	bounded by $C(1+x^{-1})$ away from the layer.
	
	The rest of this paper is organized as follows.
	In section 2, priori estimates for continuous problems will be shown in 
	each case. We use PGFEM to solve
	problem \eqref{eq1}, where we choose (either exact or approximate) 
	solutions to 
	dual problems as test functions.
	We show details related to the numerical implementation in section 3. 
	Numerical results demonstrate our scheme's efficiency 
	and the uniform first-order convergency in section 4. 
	Finally, we give a brief conclusion in section 5.
	\section{A Priori estimate}\label{sec:2}
	In this section, we will present some priori estimates for cases 
	(a)(b)(c) respectively.	First we briefly recall some results from previous 
	work for cases (a) and (b). We will prove our estimates for case (c) 
	later. 
	\subsection{Exponential boundary layer}
	Suppose the velocity coefficient $p(x)\geq p_0>0$ (or otherwise, it has a 
	negative upper bound). Equation \eqref{eq1} is now written as:
	\begin{equation}
		\label{eq1a}
		\left\{
		\begin{aligned}
			& Lu\equiv -\varepsilon u''+p(x)u'+b(x)u=f(x),\quad -1<x<1,\\
			& u(-1) = u_L, \quad u(1) = u_R,\\
			& p(x)\geq p_0>0,\quad b(x)\geq b_0\geq 0.\\
		\end{aligned}\right.
	\end{equation}
	The solution to \eqref{eq1a} admits a boundary layer at $x=1$ (at $x=-1$ if 
	$p(x)\leq p_0<0$), and it is 
	shown\cite{kellogg1978analysis, berger1984priori}	that the following 
	estimates 
	hold:
	\begin{equation}
		\label{estimate1a}
		|u^{(k)}(x)|\leq C\Big(1+\varepsilon^{-k}\exp(-\frac{\eta 
			(1-x)}{\varepsilon})\Big),\quad x\in(-1,1),\quad k=0,1,2,\cdots,
	\end{equation}
	where $C,\eta$ are positive constants independent of $\varepsilon$. We have 
	the following property at once:
	\begin{proposition}
		\rm\label{prop:exp}
		Suppose $u$ is the solution to \eqref{eq1a} and $p(x)$ is lower 
		bounded, then there exists a constant $C$ 
		independent of $\varepsilon$, such that
		\begin{equation}
			\label{eq:prop:exp}
			|(1-x)u'(x)|\leq C,\quad \forall x\in(-1,1).
		\end{equation}
	\end{proposition}
	\subsection{Cusp-like interior layer}
	Suppose there is only one turning point $x=0$, and equation 
	\eqref{eq1} reads:
	\begin{equation}
		\label{eq1b}
		\left\{
		\begin{aligned}
			& Lu\equiv-\varepsilon u''+p(x)u'+b(x)u=f(x),\quad -1<x<1,\\
			& u(-1) = u_L,\quad u(1) = u_R,\\
			& p(0)=0, p'(0)<0,|p'(x)|\geq\frac{1}{2}|p'(0)|,\quad b(x)\geq b_0> 
			0.\\
		\end{aligned}\right.
	\end{equation}
	In some papers, $x=0$ is called \emph{an attractive turning point} because 
	flows on both sides are toward the turning point, and \eqref{eq1b} is 
	called \emph{an attractive turning point problem}. Such problems are
	characterized by the parameter $\lambda=-b(0)/p'(0)$. It is 
	shown\cite{abrahamson1977priori,berger1984priori} 
	that the solution has an interior layer when $\lambda\in(0,1]$, and
	the following estimates hold:
	\begin{equation}
		\label{estimate1b}
		|u^{(k)}(x)|\leq C\big(|x|+\sqrt{\varepsilon}\big)^{\lambda-k},\quad 
		x\in 
		(-1,1),\quad k=0,1,2,\cdots.
	\end{equation}
	Similar to the previous case, the first derivative of the solution turns 
	out bounded after multiplying a factor $x$:
	\begin{proposition}
		\rm\label{prop:cusp}
		Suppose $u$ is the solution to \eqref{eq1b}, then there exists a 
		constant $C$ independent of $\varepsilon$, such that
		\begin{equation}
			\label{eq:prop:cusp}
			|xu'(x)|\leq C,\quad \forall x\in(-1,1),
		\end{equation}
		on the assumption that $\lambda\in(0,1]$.
	\end{proposition}
	If $p'(0)>0$, the problem is also called \emph{a repulsive turning point 
		problem}, and its solution is smooth near the turning point. Thus we 
		need 
	no 
	additional treatment when dealing with such turning points. 
	\subsection{Boundary turning point problem}
	Consider the turning point is positioned at an endpoint. We set the 
	interval as $[0,1]$, and \eqref{eq1} becomes:
	\begin{equation}
		\label{eq1c}
		\left\{
		\begin{aligned}
			& Lu=-\varepsilon u''+p(x)u'+b(x)u=f(x),\quad 0<x<1,\\
			& u(0) = u_L,\quad u(1) = u_R,\\
			& p(0)=0,|p'(x)|\geq\frac{1}{2}|p'(0)|,\quad b(x)\geq b_0> 0.\\
		\end{aligned}\right.
	\end{equation}
	For multiple boundary turning point problems, i.e., $p^{(k)}(0)=0$ for 
	$k=1,2,\cdots,m$, it is 
	proved\cite{vulanovic1993continuous} that there 
	exist positive constants $C,\eta$ independent of $\varepsilon$ such that 
	the 
	following estimates hold true:
	\begin{equation}
		\label{eq:mbtpp}
		|u^{(k)}(x)|\leq C\bigg(
		1+\varepsilon^{-k/2}\exp\big(-\frac{\eta 
			x}{\sqrt{\varepsilon}}\big)
		\bigg),\quad x\in(0,1),\quad k=0,1,2,\cdots.
	\end{equation}
	One can deduce the following result immediately:
	\begin{equation}
		\label{eq:btpp}\begin{aligned}
			|u^{(k)}(x)|&\leq C\min\{(1+\varepsilon^{-k/2}),(1+x^{-k})\}\\
			&=C(1+(\max\{x,\sqrt{\varepsilon}\})^{-k}).\\
		\end{aligned}
	\end{equation}
	
	Now assume that the boundary turning point is single, i.e., $p'(0)\ne0$. 	
	Boundary behaviors of such problems differ from those of \eqref{eq1a}. We 
	introduce 
	the 
	following approximated problem:
	\begin{equation}
		\label{eq1c'}
		\left\{
		\begin{aligned}
			& \tilde{L}u\equiv-\varepsilon u''+p'(0) xu'+b(0)u=f(0), \quad 
			0<x<1,\\
			& u(0) = u_L,\quad u(1) = u_R,\\
			& p'(0)\neq0,\quad b(0)>0.\\
		\end{aligned}\right.
	\end{equation}
	We divide the discussion of problem \eqref{eq1c'} into two cases by the 
	signal of 
	$p'(0)$. 
	Inspired by \cite{berger1984priori}, we could represent the solution as a 
	linear combination of Weber's parabolic cylinder functions, which we use to 
	analyze the bounds of the derivatives.
	\subsubsection{Preparations for estimates}
	We introduce a lemma to estimate the solution $u$ more precisely.	
	\begin{lemma}
		\rm\label{lemma:pDu}
		There exists one and only one solution $u$ to \eqref{eq1c}. Besides, 
		there is a 
		constant $C$ independent with
		$\varepsilon$, such that
		\begin{equation}
			\label{eq:pDu}
			|p(x)u'(x)|\leq C, \quad\forall x\in[0,1/2].
		\end{equation}
	\end{lemma}
	\begin{proof}
		We suffice to show that $\varepsilon u''(x)$ is bounded by $C$ on 
		$[0,1/2]$, for $u$ could be bounded by $f$ using the maximum 
		principle.
		First assume that $p'(0)=-\alpha<0$, and let $z(x)=u''(x)$. 
		Differentiate \eqref{eq1c} once, and we have:
		$$-\varepsilon z'+p(x)z=s(x),$$
		where $s(x)=s_1(x)+s_2(x)$,
		$$s_1(x)=f'(x)-b'(x)u,$$
		$$s_2(x)=-(p'(x)+b(x))u'.$$
		Let $$P(x)=-\int_{0}^x p(t)dt.$$
		Since $$p'(x)\leq -\frac{1}{2}\alpha<0,$$
		it holds that:
		\begin{displaymath}
			\begin{aligned}
				p(x)\leq-\frac{1}{2}\alpha x\leq0,\\
				P(x)\geq \frac{\alpha}{4}x^2\geq0.\\
			\end{aligned}
		\end{displaymath}
		Thus 
		\begin{displaymath}
			\begin{aligned}
				-\frac{P(x)}{\varepsilon}\leq0,\\
				P(t)-P(x)=\int_{t}^{x}p(\tau)d\tau\leq0,\\
			\end{aligned}
		\end{displaymath}
		for $0<t<x.$
		By variance of constants we have
		\begin{equation}
			\label{eq:z(x)}
			z(x)=z(0)\exp\Big(-\frac{P(x)}{\varepsilon}\Big)-
			\varepsilon^{-1}\int_{0}^{x}\exp\Big(\frac{P(t)-P(x)}{\varepsilon}\Big)s(t)dt.
		\end{equation}
		Taking $x=0$ in \eqref{eq1c},
		$$z(0)=\varepsilon^{-1}(b(0)u(0)-f(0))=C\varepsilon^{-1}.$$
		The integral in the second term of \eqref{eq:z(x)} is split into terms 
		with $s_1$ and 
		$s_2$:
		$$I_1=\int_{0}^{x}\exp\Big(\frac{P(t)-P(x)}{\varepsilon}\Big)s_1(t)dt\leq\int_{0}^{x}1\cdot
		Cdt=C;$$
		$$\begin{aligned}
			I_2 
			&=\int_{0}^{x}\exp\Big(\frac{P(t)-P(x)}{\varepsilon}\Big)s_2(t)dt\\
			&=-\int_{0}^{x}\exp\Big(\frac{P(t)-P(x)}{\varepsilon}\Big)\Big(p'(t)+b(t)\Big)u'(t)dt\\
			&=-\exp\Big(\frac{P(t)-P(x)}{\varepsilon}\Big)\Big(p'(t)+b(t)\Big)u(t)\bigg|_{t=0}^{t=x}\\
			&\quad+\int_{0}^{x}\exp\Big(\frac{P(t)-P(x)}{\varepsilon}\Big)\Big(p''(t)+b'(t)\Big)u(t)dt\\
			&\quad+\int_{0}^{x}\frac{d}{dt}\Big[\exp\Big(\frac{P(t)-P(x)}{\varepsilon}\Big)\Big]\Big(p'(t)+b(t)\Big)u(t)dt\\
			&\leq 
			C+C+C\int_{0}^{x}\frac{d}{dt}\Big[\exp\Big(\frac{P(t)-P(x)}{\varepsilon}\Big)\Big]dt\\
			&=C,\\
		\end{aligned}$$
		where we use the second mean value theorem for integrals before the 
		inequality. Thus we have shown that $\varepsilon u''(x)$ is 
		bounded on $[0,1]$ if $p'(0)<0$.
		
		In the case $p'(0)=\alpha>0$, the same argument could be repeated 
		with the 
		following modification:
		$$	
		z(x)=\exp\Big(\frac{P(\frac{1}{2})-P(x)}{\varepsilon}\Big)z\Big(\frac{1}{2}\Big)-
		\varepsilon^{-1}\int_{x}^{1/2}
		\exp\Big(\frac{P(\tau)-P(x)}{\varepsilon}\Big)s(\tau)d\tau,$$
		where $|z(\frac{1}{2})|$ 
		is bounded by $C$.
		The result here applies only for 
		$x\in[0,\frac{1}{2}]$ because \eqref{eq1c} has a boundary layer at 
		$x=1$.
	\end{proof}
	We have shown that $p(x)u'(x)$ is bounded near the singular point $x=0$. 
	The original problem \eqref{eq1c} is solved by decomposing $u=u_1+u_2+u_0$:
	\begin{equation}
		\begin{aligned}
			&u_0=\frac{f(0)}{b(0)},\\
			&\left\{
			\begin{aligned}
				&\tilde{L}u_1=0,\quad 0<x<1,\\
				&u_1(0)=u_L-u_0,\quad  u_1(1)=u_R-u_0,\\
			\end{aligned}
			\right.\\
			&\left\{
			\begin{aligned}
				&\tilde{L}u_2=g(x),\quad 0<x<1,\\
				&u_2(0)=u_2(1)=0,\\
			\end{aligned}
			\right.
		\end{aligned}
		\label{eq:decomposition_of_u}
	\end{equation}
	where $\tilde{L}$ is defined in \eqref{eq1c'}, and 
	$$g(x)\equiv\big(f(x)-f(0)\big)-\big(p(x)-p'(0) 
	x\big)u'-\big(b(x)-b(0)\big)u.$$
	From Lemma \ref{lemma:pDu} we could write $$g(x)=h(x)x.$$ 
	We solve $u_1$ by direct representation with parabolic cylinder functions, 
	while $u_2$ is related to Green's function for $\tilde{L}$.
	\subsubsection{Basic property of parabolic cylinder 
	function}\label{sec:2.3.2}
	Parabolic cylinder functions $U(a,x),V(a,x)$, using Weber's 
	notations, are linear independent solutions to the equation:
	\begin{equation}
		-y''+\Big(a+\frac{x^2}{4}\Big)y=0,
		\notag
		\label{eq:pcf}
	\end{equation}
	where $a$ is a coefficient. 
	The following properties will be used later:
	\begin{align}
		&\pi V(a,x)=\Gamma\Big(\frac{1}{2}+a\Big)\Big(\sin \pi a\cdot 
		U(a,x)+U(a,-x)\Big);\tag{\ref{eq:pcf}{.1}}\label{eq:imag1}\\
		&\Gamma\Big(\frac{1}{2}+a\Big)U(a,x)=\pi \sec^2 \pi 
		a\Big(V(a,-x)-\sin \pi a\cdot V(a,x))\Big);\tag{\ref{eq:pcf}{.2}}\\
		&\sqrt{2\pi}U(a,ix)=\Gamma\Big(\frac{1}{2}-a\Big)\Big(e^{-i\pi(-\frac{a}{2}-\frac{1}{4})}U(-a,x)+e^{i\pi(-\frac{a}{2}-
			\frac{1}{4})}U(-a,-x)\Big);\tag{\ref{eq:pcf}{.3}}\label{eq:imag2}\\
		&U'(a,x)+\frac{1}{2}xU(a,x)+\Big(a+\frac{1}{2}\Big)U(a+1,x)=0;\tag{\ref{eq:pcf}{.4}}\\
		&U'(a,x)-\frac{1}{2}xU(a,x)+U(a-1,x)=0;\tag{\ref{eq:pcf}{.5}}\\
		&V'(a,x)+\frac{1}{2}xV(a,x)-V(a+1,x)=0;\tag{\ref{eq:pcf}{.6}}\\
		&V'(a,x)-\frac{1}{2}xV(a,x)-\Big(a-\frac{1}{2}\Big)V(a-1,x)=0;\tag{\ref{eq:pcf}{.7}}\\
		&U(a,x)= \exp\Big(-\frac{1}{4}x^2\Big)x^{-a-\frac{1}{2}}\delta_1, 
		\quad	x\geq C_0;\tag{\ref{eq:pcf}{.8}}\\
		&V(a,x)= 
		\sqrt{\frac{2}{\pi}}\exp\Big(\frac{1}{4}x^2\Big)x^{a-\frac{1}{2}}\delta_2,
		\quad	 x\geq C_0.\tag{\ref{eq:pcf}{.9}}\label{prop:pcf}
	\end{align}
	$C_0=O(1)$ is a constant related to $a$, and coefficients 
	$\delta_{1},\delta_2$ satisfy $$|\delta_i-1|\leq \frac{1}{3}.$$
	\subsubsection{Green's function of the operator $\tilde{L}$}
	Denote $\mu_0$ by the solution to 
	$$\left\{\begin{aligned}
		&\tilde{L}u=0,\\
		&u(0)=0,\quad u(1)=1,\\
	\end{aligned}\right.$$
	and $\mu_1$ by the solution to
	$$\left\{\begin{aligned}
		&\tilde{L}u=0,\\
		&u(0)=1,\quad u(1)=0.\\
	\end{aligned}\right.$$
	The Wronskian of $\mu_0$ and $\mu_1$ is 
	$$W(x)=W(\mu_0,\mu_1)=\mu_0(x)\mu_1'(x)-\mu_1(x)\mu_0'(x).$$
	The Green's function of $\tilde{L}$ is piecewise defined on $[0,1]$:
	\begin{displaymath}
		G(x,\tau)=\left\{
		\begin{aligned}
			&-\varepsilon^{-1}\mu_0(x)\mu_1(\tau)\exp\Big(\frac{p'(0)}{2\varepsilon}(1-\tau^2)\Big)/W(1),&0\leq
			x\leq\tau\leq1,\\
			&-\varepsilon^{-1}\mu_1(x)\mu_0(\tau)\exp\Big(\frac{p'(0)}{2\varepsilon}(1-\tau^2)\Big)/W(1),&0\leq
			\tau\leq x\leq1,\\
		\end{aligned}
		\right.
	\end{displaymath}
	which satisfies 
	$$\tilde{L}G(x,\tau)=\delta(x-\tau),$$
	$$G(0)=G(1)=0.$$ 
	$u_2$ defined in \eqref{eq:decomposition_of_u} could be represented as 
	$$u_2(x)=\int_{0}^{1} G(x,\tau)\tau h(\tau) d\tau.$$
	Thus estimates for $u_2$ turn into ones for $G(\cdot,\tau)$ and its 
	derivatives.	
	\subsubsection{The case $p'(0)=\alpha>0$}
	We first estimate derivatives of the solution $u$ in the next 
	lemma.
	\begin{lemma}
		\rm\label{lemma:positive}
		Assume $u$ is the solution to \eqref{eq1c} with $p'(0)>0$, and 
		$\rho=C_0\sqrt{\varepsilon}$, where $C_0$ is defined in section 
		\ref{sec:2.3.2}. For 
		$k=1,2,\cdots$,
		\begin{equation}
			\label{eq:positive}
			|u^{(k)}(x)|\leq\left\{
			\begin{aligned}
				&C(1+\rho^{-k}), &0\leq x\leq\rho,\\
				&C(1+x^{-k}),&\rho\leq x\leq 1/2.\\
			\end{aligned}
			\right.
		\end{equation}
		Rewriting these estimates in a more compact form, we have:
		\begin{equation}
			\tag{\ref{eq:positive}{'}}
			|u^{(k)}(x)|\leq C(1+(\max\{x,\rho\})^{-k}),\quad 0\leq x\leq 1/2.
		\end{equation}
	\end{lemma}
	
	\begin{proof}
		Since $u_0$ in \eqref{eq:decomposition_of_u} has no contribution to the 
		derivative, for the sake of simplicity, we replace $u$ by $u_1+u_2$ and 
		still denote it by $u$.
		We introduce the following change of variable, for both $u_1$ and $u_2$ 
		satisfy an equation of $\tilde{L}$:
		$$\tilde{x}=\frac{x}{\sqrt{\varepsilon/\alpha}},$$
		$$u(x)=\tilde{u}(\tilde{x})\exp\Big(\frac{\tilde{x}^2}{4}\Big).$$
		Denoting $\beta=b(0)/\alpha>0$, we obtain the equation for $\tilde{u}$:
		$$-\tilde{u}''+\Big(\frac{\tilde{x}^2}{4}+\beta-\frac{1}{2}\Big)\tilde{u}=0.$$
		$\tilde{u}$ admits linear independent solutions 
		$U(\beta-\frac{1}{2},\tilde{x}),V(\beta-\frac{1}{2},\tilde{x})$. Hence
		$$u(x)=c_1\exp\Big(\frac{\tilde{x}^2}{4}\Big)U\Big(\beta-\frac{1}{2},\tilde{x}\Big)+
		c_2\exp\Big(\frac{\tilde{x}^2}{4}\Big)V\Big(\beta-\frac{1}{2},\tilde{x}\Big).$$
		Coefficients $c_1,c_2$ are determined by boundary conditions:
		\begin{displaymath}
			\left(\begin{array}{cc}
				U\Big(\beta-\frac{1}{2},0\Big) & 
				V\Big(\beta-\frac{1}{2},0\Big)\\
				\exp(\frac{\alpha}{4\varepsilon})U\Big(\beta-\frac{1}{2},\frac{1}{\sqrt{\varepsilon/\alpha}}\Big)
				& 
				\exp(\frac{\alpha}{4\varepsilon})V\Big(\beta-\frac{1}{2},\frac{1}{\sqrt{\varepsilon/\alpha}}\Big)
				\\
			\end{array}\right)
			\left(\begin{array}{c}
				c_1\\c_2\\
			\end{array}
			\right)
			=\left(\begin{array}{c}
				u_L\\u_R\\
			\end{array}
			\right).
		\end{displaymath}
		Call the matrix on the left-hand-side $A$. Denoting $K_i$ as 
		constants of $O(1)$, we could rewrite $A$ in 
		the asymptotic form:
		\begin{displaymath}
			A\approx\left(\begin{array}{cc}
				K_1 & K_2\\
				K_3\varepsilon^{\beta/2} & 
				K_4\exp(\frac{\alpha}{2\varepsilon})\varepsilon^{(1-\beta)/2} \\
			\end{array}\right),
		\end{displaymath}
		\begin{displaymath}
			A^{-1}\approx\left(\begin{array}{cc}
				K_4 & 
				-K_2\exp(-\frac{\alpha}{2\varepsilon})\varepsilon^{(\beta-1)/2}\\
				-K_3\exp(-\frac{\alpha}{2\varepsilon})\varepsilon^{\beta-1/2} & 
				K_1\exp(-\frac{\alpha}{2\varepsilon})\varepsilon^{(\beta-1)/2} 
				\\
			\end{array}\right).
		\end{displaymath}
		Therefore the coefficients are represented as follows:
		\begin{displaymath}
			\begin{aligned}
				\left(\begin{array}{c}
					c_1\\c_2\\
				\end{array}
				\right)&=A^{-1}\left(\begin{array}{c}
					u_L\\u_R\\
				\end{array}
				\right).
			\end{aligned}
		\end{displaymath}
		We omit all the constants of $O(1)$ for simplicity. Considering 
		$u_L,u_R=O(1)$, the derivatives are in the following 
		form:
		\begin{displaymath}
			\begin{aligned}
				u^{(k)}(x)&=\varepsilon^{-k/2}\bigg[ 
				c_1\exp\Big(\frac{\tilde{x}^2}{4}\Big)\Pi^k_{i=1}(-\beta-i+1)U\Big(\beta+k-\frac{1}{2},\tilde{x}\Big)
				+c_2\exp\Big(\frac{\tilde{x}^2}{4}\Big)V\Big(\beta+k-\frac{1}{2},\tilde{x}\Big)
				\bigg]\\
				&=\varepsilon^{-k/2}\bigg[c_1'\exp\Big(\frac{\tilde{x}^2}{4}\Big)U\Big(\beta+k-\frac{1}{2},\tilde{x}\Big)
				+c_2\exp(\frac{\tilde{x}^2}{4})V\Big(\beta+k-\frac{1}{2},\tilde{x}\Big)
				\bigg].
			\end{aligned}
		\end{displaymath}
		For convenience, write $u_1(x)=u_L\mu_1(x)+u_R\mu_0(x)$, and we 
		estimate 
		$\mu_0$ and $\mu_1$ first in order to 
		estimate $u_1$ and its derivatives on $[0,1/2]$:
		\begin{displaymath}
			\begin{aligned}
					%
					%
					%
					%
					%
				&|\mu_0^{(k)}(x)|\leq\left\{
				\begin{aligned}
					& 
					C\exp(-\frac{\alpha}{2\varepsilon})\varepsilon^{\frac{\beta-k-1}{2}},
					& |x|\leq\rho,\\
					& C\Big( 
					\exp(-\frac{\alpha}{2\varepsilon})\varepsilon^{\beta-\frac{1}{2}}x^{-\beta-k}
					+\varepsilon^{-k}\exp\big(-\frac{\alpha}{2\varepsilon}(1-x^2)\big)x^{\beta+k-1}
					\Big), &\rho\leq|x|\leq1,\\
				\end{aligned}
				\right.\\
				&|\mu_1^{(k)}(x)|\leq\left\{
				\begin{aligned}
					& C\varepsilon^{-\frac{k}{2}}, & |x|\leq\rho,\\
					& C\Big(\varepsilon^{\frac{\beta}{2}}x^{-\beta-k}+ 
					\exp\big(-\frac{\alpha}{2\varepsilon}(1-x^2)\big)\varepsilon^{\frac{\beta}{2}-k}
					x^{\beta+k-1} \Big), &\rho\leq|x|\leq1,\\
				\end{aligned}
				\right.\\
			\end{aligned}
		\end{displaymath}
		which hold for $k=0,1,2,\cdots$. Thus $$|W(1)|\geq 
		C\varepsilon^{\beta/2-1}.$$
		If $x\leq \rho$,  from $$u_1^{(k)}=u_L\mu_1^{(k)}+u_R\mu_0^{(k)},$$ we 
		have:
		$$\begin{aligned}
			|u_1^{(k)}(x)|&\leq 
			C\Big(|u_R||\mu_0^{(k)}(x)|+|u_L||\mu_1^{(k)}(x)|\Big)\\
			&\leq 
			C\Big(|u_R|\exp(-\frac{\alpha}{2\varepsilon})\varepsilon^{\frac{\beta-k-1}{2}}+
			|u_L|\varepsilon^{-\frac{k}{2}}\Big),\quad k=1,2,\cdots.\\
		\end{aligned}$$
		If $\rho\leq x\leq1/2$, 
		$$\exp(-\frac{\alpha}{2\varepsilon}(1-x^2))\leq\exp(-\frac{\alpha'}{2\varepsilon}),$$
		where $\alpha'=\frac{3}{4}\alpha,$ we have
		\begin{displaymath}
			\begin{aligned}
				|u_1^{(k)}(x)|&\leq 
				C|u_L|\Big(\varepsilon^{\frac{\beta}{2}}x^{-\beta-k}+ 
				\exp\big(-\frac{\alpha}{2\varepsilon}(1-x^2)\big)\varepsilon^{\frac{\beta}{2}-k}
				x^{\beta+k-1} \Big)\\
				&\quad +C|u_R|\Big( 
				\exp(-\frac{\alpha}{2\varepsilon})\varepsilon^{\beta-\frac{1}{2}}x^{-\beta-k}
				+\varepsilon^{-k}\exp\big(-\frac{\alpha}{2\varepsilon}(1-x^2)\big)x^{\beta+k-1}
				\Big)\\
				&\leq C(1+x^{-k}),\quad k=1,2,\cdots.
			\end{aligned}
		\end{displaymath}	
		Constant $C$ may depend on $\beta,k,\alpha,C_0$.
		Derivatives of $u_2$ are estimated by induction.
		\begin{displaymath}
			\begin{aligned}
				|u_2'(x)|&\leq C\int_{0}^{1}|G_x(x,\tau)|\tau d\tau\\
				&=\left\{ 
				\begin{aligned}
					&\int_{0}^{x}+\int_{x}^{\rho}+\int_{\rho}^{1},&0\leq 
					x\leq\rho,\\
					&\int_{0}^{\rho}+\int_{\rho}^{x}+\int_{x}^{1},&\rho\leq 
					x\leq\frac{1}{2}.\\
				\end{aligned}
				\right.
			\end{aligned}
		\end{displaymath}
		Denote these six integrals as $I_{1,2,3}$ and $I_{1,2,3}'$.	The first 
		three are easy to handle:\\
		$I_1: 0\leq\tau\leq x\leq\rho$:\\ 
		$$\begin{aligned}
			&|G_x|\leq 
			C\varepsilon^{-1}\exp\Big(-\frac{\alpha}{2\varepsilon}\tau^2\Big),\\
			&I_1=\int_{0}^x |G_x|\tau d\tau\leq C.\\
		\end{aligned}$$			
		$I_2: 0\leq x\leq \tau\leq\rho$:\\
		$$\begin{aligned}
			&|G_x|\leq 
			C\varepsilon^{-1}\exp\Big(-\frac{\alpha}{2\varepsilon}\tau^2\Big),\\
			&I_2=\int_{x}^{\rho} |G_x|\tau d\tau\leq C.\\
		\end{aligned}$$
		$I_3: 0\leq x\leq \rho\leq \tau$:\\
		$$\begin{aligned}
			&|G_x|\leq 
			C\Big(\varepsilon^{\frac{\beta}{2}-1}\exp(-\frac{\alpha}{2\varepsilon}\tau^2)\tau^{-\beta}+
			\varepsilon^{\frac{\beta}{2}-1}\exp(-\frac{\alpha}{2\varepsilon})\tau^{\beta-1}\Big),\\
			&I_3=\int_{\rho}^{1} |G_x|\tau d\tau\leq C\int_{\rho}^{1}
			\exp\Big(-\frac{\alpha}{2\varepsilon}\tau^2\Big)\tau^{1-\beta}\varepsilon^{\frac{\beta-1}{2}}
			\varepsilon^{-\frac{1}{2}}d\tau+C\exp\Big(-\frac{\alpha}{2\varepsilon}\Big)\varepsilon^{\frac{\beta}{2}-1}\\
			&\quad \leq C\varepsilon^{-\frac{1}{2}}+C\\
			&\quad =C(1+\rho^{-1}).\\
		\end{aligned}$$
		Thus when $x\leq\rho$, we have $$|u_2'(x)|\leq C(1+\rho^{-1}).$$
		For $x\geq\rho$, consider the last three integrals:\\
		$I_1': 0\leq\tau\leq\rho\leq x\leq\frac{1}{2}$: 
		$$\begin{aligned}
			|G_x|&\leq\varepsilon^{\frac{\beta-1}{2}}\exp\Big(-\frac{\alpha}{2\varepsilon}\tau^2\Big)x^{-\beta-1}
			+\varepsilon^{\frac{\beta-1}{2}}x^{\beta}\exp\Big(-\frac{\alpha}{2\varepsilon}(1-x^2)\Big)\exp\Big(-\frac{\alpha}{2\varepsilon}\tau^2\Big)\varepsilon^{-1}\\
			&\leq 
			C\varepsilon^{-1}\exp\Big(-\frac{\alpha}{2\varepsilon}\tau^2\Big),\\
			I_1'&=\int_{0}^{\rho} |G_x|\tau d\tau\leq C.\\
		\end{aligned}$$
		$I_2': \rho\leq\tau\leq x\leq\frac{1}{2}$:
		$$\begin{aligned}
			|G_x|&\leq 
			\bigg(x^{-\beta-1}+\varepsilon^{-1}\exp\big(-\frac{\alpha}{2\varepsilon}(1-x^2)\big)x^{\beta}
			\bigg)\bigg(
			\exp\big(-\frac{\alpha}{2\varepsilon}\tau^2\big)\varepsilon^{\beta-\frac{1}{2}}\tau^{-\beta}+\tau^{\beta-1}
			\bigg),\\
			|G_x|\tau&\leq 
			\bigg(x^{-\beta-1}+\varepsilon^{-1}\exp\big(-\frac{\alpha}{2\varepsilon}(1-x^2)\big)x^{\beta}\bigg)
			\bigg(\exp\big(-\frac{\alpha}{2\varepsilon}\tau^2\big)\varepsilon^{\beta-\frac{1}{2}}+\tau^\beta\bigg)\\
			&\leq\Big(x^{-\beta-1}+\varepsilon^{-1}
			\exp\big(-\frac{\alpha'}{2\varepsilon}\big)\Big)\Big(C\varepsilon^{\frac{\beta}{2}}+\tau^\beta\Big),\\
			I_2'&=\int_{\rho}^x |G_x|\tau d\tau\leq C(1+x^{-1}).\\
		\end{aligned}$$
		$I_3': \rho\leq x\leq\tau\leq1,x\leq\frac{1}{2}$:
		$$\begin{aligned}
			&|G_x|\leq\bigg(\varepsilon^{\beta-\frac{1}{2}}\exp\big(-\frac{\alpha}{2\varepsilon}\big)x^{-\beta-1}
			+\varepsilon^{-1}\exp\big(-\frac{\alpha}{2\varepsilon}(1-x^2)\big)x^{\beta}
			\bigg)\\
			&\quad\quad\quad\cdot\bigg(
			\exp\big(\frac{\alpha}{2\varepsilon}(1-\tau^2)\big)\tau^{-\beta}+\tau^{\beta-1}
			\bigg),\\
			&I_3'=\int_{x}^1 |G_x|\tau d\tau\leq C(1+x^{-1}).\\
		\end{aligned}$$
		Therefore we have $$|u_2'(x)|\leq C(1+x^{-1}),$$
		when $x\geq\rho$.
		
		Considering that the behavior near $x=1$
		has been studied well as an exponential boundary layer, we narrow down 
		the priori estimate for $u$ to $[0,1/2]$:
		\begin{displaymath}
			|u'(x)|\leq\left\{
			\begin{aligned}
				&C(1+\rho^{-1}), &0\leq x\leq\rho,\\
				&C(1+x^{-1}),&\rho\leq x\leq 1/2.\\
			\end{aligned}
			\right.
		\end{displaymath}
		
		For higher derivatives $u^{(k)}(k\geq2)$, we 
		could 
		differentiate $(k-1)$ 
		times the equation in \eqref{eq1c} and split 
		$v=u^{(k-1)}$ into three parts as $u$ is decomposed in 
		\eqref{eq:decomposition_of_u}. It is noticed that $v$ satisfies a 
		similar equation 
		to 
		\eqref{eq1c'}, with $\beta$ replaced by $\beta+k-1$ and boundary 
		conditions in 
		asymptotic forms. Actually, we can obtain by induction, for 
		$k=1,2,\cdots$:
		$$\begin{aligned}
			&u^{(k)}(0)=\varepsilon^{-\frac{k}{2}},\\
			&u^{(k)}(1)=\varepsilon^{-k}.\\
		\end{aligned}$$
		Simple calculations yield the following result:
		\begin{displaymath}
			|u^{(k)}(x)|\leq\left\{
			\begin{aligned}
				&C(1+\rho^{-k}), &0\leq x\leq\rho,\\
				&C(1+x^{-k}),&\rho\leq x\leq 1/2.\\
			\end{aligned}
			\right.
		\end{displaymath}
	\end{proof}
	\begin{remark}
		\rm\label{remark_positive}
		If we take $x\approx1$ into consideration, the estimates could be 
		modified as:
		\begin{equation}
			|u^{(k)}(x)|\leq\left\{
			\begin{aligned}
				&C(1+\rho^{-k}), &0\leq x\leq\rho,\\
				&C\Big(1+x^{-k}+(1-x)^{-k}\Big),&\rho\leq x\leq 1;\\
			\end{aligned}
			\right.\quad k=1,2,\cdots.
		\end{equation}
		Note that for $\frac{1}{2}\leq x\leq1,$
		$$\exp(-\frac{\alpha}{2\varepsilon}(1-x^2))\leq 
		C(\frac{\varepsilon}{1-x})^k.$$
	\end{remark}
	The following proposition holds as a direct conclusion, and we will use 
	these propositions to show the convergence of the numerical 
	method.
	\begin{proposition}
		\rm\label{prop:positive}
		For the solution $u$ of \eqref{eq1c}, when $p'(0)>0$, there exists a 
		constant 
		$C$, satisfying
		\begin{equation}
			\label{eq:prop:positive}
			|xu'(x)|\leq C, \quad x\leq1/2,
		\end{equation}
		\begin{equation}
			\label{eq:prop:positive2}
			\tag{\ref{eq:prop:positive}{'}}
			|(1-x)u'(x)|\leq C, \quad x\geq1/2.
		\end{equation}	
	\end{proposition}
	\subsubsection{The case $p'(0)=-\alpha<0$}
	In this case, estimates for derivatives of $u$ are mildly different.
	\begin{lemma}
		\rm\label{lemma:negative} 
		Assume that $u$ is the solution to \eqref{eq1c} when $p'(0)<0$. Let 
		$\beta=b(0)/\alpha>0$, and $\rho$ is 
		defined as above. Then we have 
		the following estimates:
		\begin{equation}
			\label{eq:negative}
			\begin{aligned}
				|u^{(k)}(x)|&\leq\left\{
				\begin{aligned}
					&C\Big(1+\rho^{\beta-k}+|u_L|\rho^{-k}\Big),&0\leq 
					x\leq\rho,\\
					&C\Big(1+x^{\beta-k}+|u_L|x^{-k}\Big),&\rho\leq x\leq1;\\
				\end{aligned}
				\right.\quad k=1,2,\cdots,\\
			\end{aligned}
		\end{equation}
		or otherwise in a compact form:
		\begin{equation}
			\label{eq:negative2}
			\tag{\ref{eq:negative}{'}}
			|u^{(k)}(x)|\leq 
			C(1+(\max\{x,\rho\})^{\beta-k}+|u_L|(\max\{x,\rho\})^{-k}),\quad 
			0\leq 
			x\leq 
			1,\quad k=1,2,\cdots.
		\end{equation}
	\end{lemma}
	\begin{proof}
		If we let 
		$\tilde{x}=\frac{x}{\sqrt{-\varepsilon/\alpha}}$, the new variable 
		becomes pure imaginary. 
		From the property that evaluation of $U(a,iz)$ and $V(a,iz)$ could be 
		represented by 
		$U(-a,z),V(-a,z)$ (c.f. \eqref{eq:imag1},\eqref{eq:imag2}), we assume 
		the 
		solution to be:
		$$u(x)=c_1\exp\Big(\frac{\tilde{x}^2}{4}\Big)U\Big(\frac{1}{2}+\beta,|\tilde{x}|\Big)+
		c_2\exp\Big(\frac{\tilde{x}^2}{4}\Big)V\Big(\frac{1}{2}+\beta,|\tilde{x}|\Big).$$
		For simplicity we denote 
		$\hat{x}=|\tilde{x}|=\frac{x}{\sqrt{\varepsilon/\alpha}}$, which gives:
		$$u(x)=c_1\exp\Big(-\frac{\hat{x}^2}{4}\Big)U\Big(\frac{1}{2}+\beta,\hat{x}\Big)+
		c_2\exp\Big(-\frac{\hat{x}^2}{4}\Big)V\Big(\frac{1}{2}+\beta,\hat{x}\Big).$$
		Again we solve coefficients $c_1,c_2$ from boundary conditions 
		$u_L,u_R$. Let
		\begin{displaymath}
			\begin{aligned}
				A&=\left(\begin{array}{cc}
					U\Big(\frac{1}{2}+\beta,0\Big) & 
					V\Big(\frac{1}{2}+\beta,0\Big)\\
					\exp(-\frac{\alpha}{4\varepsilon})U\Big(\frac{1}{2}+\beta,\frac{1}{\sqrt{\varepsilon/\alpha}}\Big)
					& 
					\exp(-\frac{\alpha}{4\varepsilon})V\Big(\frac{1}{2}+\beta,\frac{1}{\sqrt{\varepsilon/\alpha}}\Big)
					\\
				\end{array}\right)\\
				&\approx\left(\begin{array}{cc}
					K_1 & K_2\\
					K_3\exp(-\frac{\alpha}{2\varepsilon})\varepsilon^{(1+\beta)/2}
					& K_4\varepsilon^{-\beta/2} \\
				\end{array}\right).\\
			\end{aligned}
		\end{displaymath}
		The inverse of $A$ reads:
		\begin{displaymath}
			\begin{aligned}
				A^{-1}&\approx\left(\begin{array}{cc}
					K_4 & -K_2\varepsilon^{\beta/2}\\
					-K_3\exp(-\frac{\alpha}{2\varepsilon})\varepsilon^{1/2+\beta}
					& K_1\varepsilon^{\beta/2} \\
				\end{array}\right)\\
				&\approx C\left(\begin{array}{cc}
					1&\varepsilon^{\beta/2}\\
					\exp(-\frac{\alpha}{2\varepsilon})\varepsilon^{1/2+\beta}&\varepsilon^{\beta/2}\\
				\end{array}\right).
			\end{aligned}
		\end{displaymath}
		Then $c_1,c_2$ could be written as
		\begin{displaymath}
			\begin{aligned}
				\left(\begin{array}{c}
					c_1\\c_2\\
				\end{array}
				\right)&=A^{-1}\left(\begin{array}{c}
					u_L\\u_R\\
				\end{array}
				\right).\\
			\end{aligned}
		\end{displaymath}
		As the previous case, for $k=1,2,\cdots$,
		$$u^{(k)}(x)=\varepsilon^{-k/2}\Big[ 
		c_1'\exp\Big(-\frac{\hat{x}^2}{4}\Big)U\Big(\beta+\frac{1}{2}-k,\hat{x}\Big)+
		c_2'\exp\Big(-\frac{\hat{x}^2}{4}\Big)V\Big(\beta+\frac{1}{2}-k,\hat{x}\Big)
		\Big].$$
		For $\mu_0,\mu_1$, and $k=0,1,2,\cdots$,
		\begin{displaymath}
			\begin{aligned}
				&\mu_0^{(k)}(x)=\exp\Big(-\frac{\hat{x}^2}{4}\Big)\varepsilon^{\frac{\beta-k}{2}}
				\bigg(U\Big(\beta+\frac{1}{2}-k,\hat{x}\Big)+V\Big(\beta+\frac{1}{2}-k,\hat{x}\Big)\bigg),\\
				&\mu_1^{(k)}(x)=\exp\Big(-\frac{\hat{x}^2}{4}\Big)
				\varepsilon^{-\frac{k}{2}}\bigg(U\Big(\beta+\frac{1}{2}-k,\hat{x}\Big)
				+\varepsilon^{\frac{1}{2}+\beta}\exp(-\frac{\alpha}{2\varepsilon})V\Big(\beta+\frac{1}{2}-k,
				\hat{x}\Big)\bigg).\\
			\end{aligned}
		\end{displaymath}
		We have the following estimates:
		\begin{displaymath}
			\begin{aligned}
				&|\mu_0^{(k)}(x)|\leq\left\{
				\begin{aligned}
					&C\varepsilon^{\frac{\beta-k}{2}},&x\leq\rho,\\
					&C\Big(\exp\big(-\frac{\alpha}{2\varepsilon}x^2\big)
					\varepsilon^{\beta+\frac{1}{2}-k}x^{-\beta+k-1}+x^{\beta-k}\Big),&x\geq\rho;\\
				\end{aligned}
				\right.\\
				&|\mu_1^{(k)}(x)|\leq\left\{
				\begin{aligned}
					&C\varepsilon^{-\frac{k}{2}},&x\leq\rho,\\
					&C\Big(\exp\big(-\frac{\alpha}{2\varepsilon}x^2\big)
					\varepsilon^{\frac{\beta+1}{2}-k}x^{-\beta+k-1}
					+\exp(-\frac{\alpha}{2\varepsilon})
					\varepsilon^{\frac{\beta+1}{2}}x^{\beta-k}\Big),&x\geq\rho;\\
				\end{aligned}
				\right.
			\end{aligned}
			\quad k=0,1,\cdots.
		\end{displaymath}
		Since these two estimates are different from the case $p'(0)>0$, we 
		might keep $u_L$ and $u_R$ 
		as 
		independent variables.\\
		\begin{displaymath}
			\begin{aligned}
				|u_1^{(k)}(x)|&\leq 
				|u_R||\mu_0^{(k)}(x)|+|u_L||\mu_1^{(k)}(x)|\\
				&\leq\left\{
				\begin{aligned}
					&C\Big(|u_R|\varepsilon^{\frac{\beta-k}{2}}+|u_L|\varepsilon^{-\frac{k}{2}}\Big),&x\leq\rho,\\
					&C\Big(|u_R|x^{\beta-k}+|u_L|x^{-k}\Big),&x\geq\rho;\\
				\end{aligned}
				\right. \quad k=1,2,\cdots.\\
			\end{aligned}
		\end{displaymath}
		$u_2$ could be estimated similarly by computing integrals of Green's 
		function. We omit these details and present the following result.
		\begin{displaymath}
			\begin{aligned}
				|u_2^{(k)}(x)|&\leq\left\{
				\begin{aligned}
					&C\Big(1+\rho^{\beta-k}\Big),&x\leq\rho,\\
					&C\Big(1+x^{\beta-k}\Big),&x\geq\rho;\\
				\end{aligned}
				\right. \quad k=1,2,\cdots.\\
			\end{aligned}
		\end{displaymath}
		The conclusion in Lemma \ref{lemma:negative} consists of estimates for 
		$u_1$ and $u_2$.
	\end{proof}
	\begin{remark}
		\rm\label{remark_negative}
		The solution of the case $p'(0)<0$ is smooth at the endpoint $x=1$, 
		hence estimates are 
		made on the whole interval $[0,1]$. The result in Lemma 
		\ref{lemma:negative} is quite 
		similar to 
		\cite{berger1984priori}, except that nonzero $u_L$ might lower 
		the regularity of the solution.
	\end{remark}
	\begin{proposition}
		\rm\label{prop:negative}
		If $u$ is the solution to \eqref{eq1c} with $p'(0)<0$, there exists a 
		constant $C$ independent of $\varepsilon$, such that
		\begin{equation}
			\label{eq:prop:negative}
			|xu'(x)|\leq C.
		\end{equation}
	\end{proposition}
	\begin{remark}
		\rm\label{prop:compare}
		Estimates \eqref{eq:mbtpp} are stronger than Lemma \ref{lemma:positive} 
		and \ref{lemma:negative}. 
		Analysis in \cite{vulanovic1993continuous} applies in the cases 
		$k\ge2$, where $k$ stands for multiples of the turning point, while the 
		same argument no longer holds for $k=1$. 
		Another difference is that their estimates are made on the whole 
		interval $[0,1]$. In contrast, estimates hold for $[0,1/2]$ in Lemma 
		\ref{lemma:positive} and for the whole interval in Lemma 
		\ref{lemma:negative}. Estimates \eqref{eq:positive} and 
		\eqref{eq:negative}
		give upper bounds for $x\leq\rho$ and $x\geq\rho$ separately when the 
		turning 
		point is single, and it is unknown 
		whether these estimates could be combined into one expression in an 
		essential way.
	\end{remark}
	\section{Numerical method}\label{sec:3}
	In this section we first introduce some definitions and weak formulations 
	in section \ref{sec:3.1}. We derive the weak solution using a 
	Petrov-Galerkin 
	finite 
	element method (PGFEM), summarized in Algorithm \ref{alg_PGFEM}. 
	If we know the analytic expressions of the solutions to the dual problems, 
	we directly use them as the test functions in PGFEM; otherwise, the dual 
	problems are solved 
	numerically by TFPM on a uniform mesh, as described in section
	\ref{sec:3.2}. Furthermore, we prove 
	first-order convergency of PGFEM in $L^\infty$-norm and energy 
	norm in section
	\ref{sec:3.3} when test functions are evaluated exactly.
	\subsection{Definitions and formulations}\label{sec:3.1}
	The weak form of problem \eqref{eq1} is: Find $u\in H^1(x_L,x_R)$ such that
	\begin{equation}
		\label{weak_eq1}
		\begin{aligned}
			&A_\epsilon(u,v)\equiv \epsilon(u',v')+(pu',v)+(bu,v)=(f,v),\quad 
			\forall 
			v\in H^1_0(x_L,x_R),\\
			&u(x_L)=u_L,\quad u(x_R)=u_R.
		\end{aligned}
	\end{equation}
	
	Let us take a partition $\{x_i, i=0,1,\cdots,N\}$ on $[x_L,x_R]$, 
	including any possible interior turning point:
	$$x_L=x_0<x_1<\cdots<x_N=x_R,$$
	$$I_i=[x_{i-1},x_i],\quad i=1,2,\cdots,N,$$
	$$h_i=\left\{
	\begin{aligned}
		&x_{i}-x_{i-1},&i=1,\cdots,N,\\
		&0,&i=0,N+1,\\
	\end{aligned}
	\right.$$
	and the mesh size $h$ is defined as
	$$h=\max_{1\leq i\leq N}h_i.$$
	
	In this section, we use $L^\infty$, $L^2$ and an energy norm 
	$||\cdot||_\varepsilon$ for a function $u$:
	\begin{align}
		&||u||_{L^\infty}=\max_{x_L\leq x\leq x_R}|u(x)|,\\
		&||u||_{L^2} = \sqrt{\int_{x_L}^{x_R}|u(x)|^2dx},\\
		&||u||_\varepsilon = 
		\sqrt{||u||_{L^2}^2+\varepsilon||u'||_{L^2}^2},\label{eq:ucon}
	\end{align}
	and the corresponding discrete infinity norm $||\cdot||_{L^\infty_h}$ and 
	discrete energy norm $||\cdot||_{\varepsilon,h}$ for a grid function $u_h$:
	\begin{align}
		&||u_h||_{L^\infty_h}=\max_{0\leq i\leq N}|u_h(x_i)|,\label{eq:uinf}\\
		&||u_h||_{\varepsilon,h} = \sqrt{||u_h||_{L^2_h}^2 +\varepsilon 
			||u_h'||_{L^2_h}^2}.\label{eq:ueng}
	\end{align}
	Here $L^2_h$ is the discrete $L^2$ space with the norm defined on the grid, 
	and $u_h'$ is computed by a difference scheme:
	\begin{align}
		||u_h||_{L^2_h}&=\sqrt{\sum_{i=0}^{N}u_h^2(x_i)\frac{h_i+h_{i+1}}{2}},\label{eq:ul2}\\
		||u_h'||_{L^2_h}&=\sqrt{\sum_{i=1}^{N}\bigg(\frac{u_h(x_i)-u_h(x_{i-1})}{h_i}\bigg)^2h_i}.\label{eq:upl2}
	\end{align}
	
	Before discretization of finite element method, we first approximate 
	\eqref{eq1} by the following problem:
	\begin{equation}
		\label{eq1'}
		\left\{
		\begin{aligned}
			& \bar{L}u_h\equiv-\varepsilon 
			u_h''+\bar{p}(x)u_h'+\bar{b}(x)u_h=\bar{f}(x),\quad
			x_L<x<x_R,\\
			& u_h(x_L) = u_L, \quad u_h(x_R) = u_R,\\
		\end{aligned}\right.
	\end{equation} 
	where $\bar{p},\bar{b},\bar{f}$ are piecewise approximations to the 
	corresponding functions. Test function space $V_h$ is defined by a group of 
	basis 
	functions $\{\psi_i\}_{i=1}^{N-1}$ with $\psi_i$ solving the dual problem 
	of 
	\eqref{eq1'} on 
	$I_i\cup I_{i+1}$:
	\begin{equation}
		\left\{
		\begin{aligned}
			& \bar{L}^*\psi_i\equiv-\varepsilon 
			\psi_i''-\bar{p}(x)\psi_i'+(\bar{b}(x)-\bar{p}'(x))\psi_i=0, \quad 
			x_{i-1}<x<x_{i+1},\\
			& \psi_i(x_i) = 1, \quad\psi_i(x_j) = 0\,(j\neq i). \\
		\end{aligned}\right.
	\end{equation}
	Functions $\{\psi_i\}_{i=1}^{N-1}$ are referred to as $L^*-splines$ in some 
	articles.
	
	Then we use PGFEM to discretize the weak form of 
	\eqref{eq1'}: Find $u_h\in 
	U_h$ such that
	\begin{equation}\label{pgfem}
		\begin{aligned}
			&A_\epsilon(u_h,v_h)\equiv 
			\epsilon(u'_h,v'_h)+(\bar{p}u'_h,v_h)+(\bar{b}u_h,v_h)=(\bar{f},v_h),
			\quad \forall v_h\in V_h,\\
			&u_h(x_L)=u_L,\quad u_h(x_R)=u_R.
		\end{aligned}
	\end{equation}
	where
	\begin{align}
		U_h&=\bigg\{v_h\in C[x_L,x_R]\,\bigg|\,v_h|_{I_i} \mbox{ is linear 
		function}, i=1,\cdots,N\bigg\},\\
		V_h&=\bigg\{v_h\,\bigg|\,v_h=\sum_{i=1}^{N-1}c_i\psi_i,\,c_i\in\mathbb{R},
		\,i=1,\cdots,N-1\bigg\}.\label{def:V_h}
	\end{align}
	
	\begin{remark}
		\rm\label{remark_exact_numerical}
		If we use parabolic cylinder functions as test functions, it is 
		usual to compute a cut-off of the series expansion of these special 
		functions in order to generate the stiffness matrix and the 
		right-hand-side 
		term. We compute parabolic cylinder 
		functions in MATLAB using codes from fortran90 
		by\cite{temme2000numerical,olver2010nist}.
		In some cases, numerical cost is 
		expensive when we need these special functions to be precise enough. 
		Moreover, we could not analytically represent 
		the solution to the dual problem with a nonlinear first-order 
		coefficient.
		In practice, it works as well if we substitute exact evaluations of 
		special functions with numerical solutions described in the following 
		subsection. 
	\end{remark}
	\subsection{Numerical method of dual problems}\label{sec:3.2}
	We apply TFPM on the uniform mesh to each dual problem.
	Precisely, for a specific dual problem:
	\begin{equation}
		\label{dual}
		-\varepsilon \psi''+\hat{a}(x-X_0)\psi'+\hat{b}\psi=0,\quad X_1<x<X_2,
	\end{equation}
	the solution is determined with the following boundary conditions, for 
	instance:
	\begin{equation}
		\label{dual:bd}
		\psi(X_1)=1,\quad\psi(X_2)=0.
	\end{equation}
			%
			%
	%
	We make a uniform partition on the subinterval:
	\begin{displaymath}
		Y_j=X_1+(j-1)\frac{X_2-X_1}{N_1},\quad j=0,1,\cdots,N_1.
	\end{displaymath}
	TFPM solution on $[Y_{i-1},Y_{i+1}]$ is the linear 
	combination 
	of solutions to the equation:
	\begin{equation}
		-\varepsilon \psi''+\hat{a}(Y_{i}-X_0)\psi'+\hat{b}\psi=0,\quad 
		Y_{j-1}<x<Y_{j+1}.
		\label{dual:subinterval}
	\end{equation}
	Equation \eqref{dual:subinterval} admits two exponential solutions 
	$\psi^{(1)},\psi^{(2)}$.
	Denoting $\psi_j=\psi(Y_j)$, we suppose:
	\begin{equation}
		\label{eq:cond_dual}
		\alpha_{i,i-1}\psi_{i-1}+\alpha_{i,i}\psi_i+\alpha_{i,i+1}\psi_{i+1}=0,
	\end{equation}
	where we presume $\alpha_{i,i}=1$ because 
	\eqref{dual:subinterval} is homogeneous, and the rest of the coefficients 
	are 
	determined by 
	requiring $\psi^{(1)}$ and $\psi^{(2)}$ to satisfy \eqref{eq:cond_dual}. 
	By gathering conditions \eqref{eq:cond_dual} at $i=1,\cdots,N_1-1$ 
	together with boundary 
	conditions \eqref{dual:bd}, we obtain a tri-diagonal linear system 
	which gives evaluations of approximated dual solutions
	$\{\psi(Y_i)\}_{i=0}^{N_1}$. We remark that TFPM on the uniform mesh 
	described above yields smaller errors than simple 
	finite difference methods.
	
	To compute derivatives at $X_1$ and $X_2$, we represent the solution by
	exponential basis functions on $[Y_0,Y_1]$ and $[Y_{N_1-1},Y_{N_1}]$ 
	respectively. For instance, TFPM solution on $[Y_0,Y_1]$ is assumed to 
	satisfy \eqref{dual:subinterval} which is defined on $Y_0<x<Y_1$. By 
	boundary conditions at $Y_0$ and $Y_1$, the solution is 
	identified on $[Y_0,Y_1]$, and $\psi'(X_1)$ is available by direct 
	calculations.
	
	\subsection{Main results}\label{sec:3.3}
	For PGFEM using exact test functions, we have the following convergence 
	theorem:
	\begin{theorem}
		\rm\label{thm:Linf}
		\textbf{(First-order $L^{\infty}$ Uniform Convergence)} Assume $p,b,f$ 
		as 
		above, and singularities of type (a), (b), and (c) might occur in the 
		solution. Then PGFEM described in section \ref{sec:3.1} converges 
		uniformly 
		in $L^\infty$ norm, i.e., there exists a constant $C$ independent 
		of 
		$h,\varepsilon$, such that
		\begin{equation}
			\label{eq:thm:Linf}
			||e||_{L^\infty}\leq Ch,
		\end{equation}
		with $e=u-u_h$, where $u$ is the strong solution to problem 
		\eqref{weak_eq1}, $u_h$ is the strong solution to problem \eqref{pgfem}.
	\end{theorem}
	\begin{proof}
		We classify all the subintervals $I_k$'s into ones close 
		to singular points and ones far away. More 
		exactly, we list all singular points in ascending order as 
		$\{s_i\}_{i=1}^{m'}$ which 
		contain any possible interior turning points $z_i$ (with 
		$\lambda_i\in(0,1]$) 
		and endpoints (with a boundary layer). Let $0<\delta<\min_i 
		|s_{i}-s_{i-1}|/3$ and define:
		\begin{displaymath}
			\begin{aligned}
				&J_i=[s_i-\delta,s_i+\delta],\quad (i=1,\cdots,m')\\
				&J_r=[x_L,x_R]-\bigcup J_i.\\
			\end{aligned}
		\end{displaymath}
		According to whether the middle point of a subinterval $I_k$ locates in 
		some $J_i$ or not, we 
		approximate $p$ on $I_k$ by either linear functions or constants:
		\begin{displaymath}
			\bar{p}(x)\big|_{I_k}=\left\{
			\begin{aligned}
				&p(x_{k-\frac{1}{2}}), &\text{if } x_{k-\frac{1}{2}}\in J_r,\\
				&p'(x_k^*)(x-x_k^*)+p(x_k^*), &\text{if } x_{k-\frac{1}{2}}\in 
				J_i,\\
			\end{aligned}
			\right.
		\end{displaymath}
		where
		\begin{displaymath}
			x_k^*=\left\{
			\begin{aligned}
				&s_i, &\text{if } I_k \text{ is adjacent to some } 
				s_i,\\
				&x_{k-\frac{1}{2}},& \text{otherwise}.
			\end{aligned}
			\right.
		\end{displaymath}
		Take $\bar{b}(x)\big|_{I_k}$ as piecewise constant 
		$b(x_{k-\frac{1}{2}})$ on every subinterval	$I_k$ and 
		$\bar{f}(x)\big|_{I_k}$ 
		likewise.
		Thus test functions induced by linear approximations are 
		represented by parabolic cylinder functions and ones derived by 
		constants 
		are exponential functions.
		
		We note that maximum principle holds for $L$ and $\bar{L}$, i.e., there 
		exists a constant $C$ independent of $\varepsilon,h$:
		$$||v||_{\infty}\leq C\Big(||Lv||_{\infty}+|v(x_L)|+|v(x_R)|\Big),$$
		$$||v||_{\infty}\leq 
		C\Big(||\bar{L}v||_{\infty}+|v(x_L)|+|v(x_R)|\Big).$$
		With the same arguments by Gartland and 
		Farrell\cite{farrell1988uniform}, 
		we have:
		$$||e||_\infty\leq C\Big\{||(p-\bar{p})u'||_\infty+||b-\bar{b}||
		_\infty||u||_\infty+||f-\bar{f}||_\infty\Big\},$$
		where $e=u-u_h$. We separate the discussion by whether $x_{k-1/2}$ 
		lies in some $J_i$, where $x\in I_k$:\\[1mm]
		\textbf{Case (a)} $x_{k-1/2}\in J_i$.\\[-2mm]
		\begin{itemize}
			\item[]	Suppose 
			$x_{k-1/2}\in J_i$. If the singular point
			$s_i\in(x_L,x_R)$, we 
			have the estimate of the derivative $$|u'(x)|\leq 
			C\Big(|x-s_i|+\sqrt{\varepsilon}\Big)^{\lambda_i-1},$$
			where the coefficient $\lambda_i=-b(s_i)/p'(s_i)\in(0,1]$; 
			$|u'(x)|$ 
			is bounded by a constant $C$ for other 
			evaluations of $\lambda$. By definition of 
			$\bar{p}$, the term $|(p-\bar{p})u'|$ is under control due to the 
			following inequalities and \eqref{eq:prop:cusp}:\\[1mm]
			\begin{displaymath}
				\begin{aligned}
					&|p-\bar{p}|\leq Ch^2\leq Ch|x-s_i|, &\mbox{if }|x-s_i|\geq 
					h,\\
					&|p-\bar{p}|\leq C|x-s_i|^2\leq Ch|x-s_i|, 
					&\mbox{if }|x-s_i|\leq h.\\
				\end{aligned}
			\end{displaymath}\\[1mm]
			In both cases $|x-s_i|$ multiplied by $|u'(x)|$ is bounded. 
			The second inequality holds, for we approximate $p$ by the 
			first-order Taylor polynomial at $x=s_i$ in the intervals 
			adjacent to $s_i$.
			
			In the cases where $s_i$ is an endpoint of the whole interval, 
			generally
			the derivative $u'(x)$ might be significant near the singular 
			point, 
			such as $\varepsilon^{-1}$ or $\varepsilon^{-1/2}$. We have shown 
			in 
			\eqref{eq:prop:exp}, \eqref{eq:prop:positive}, and 
			\eqref{eq:prop:negative}
			that
			$|(x-s_i)u'(x)|$ is bounded uniformly by a constant $C$.
		\end{itemize}
		\textbf{Case (b)} $x_{k-1/2}\in J_r$.\\[-2mm]
		\begin{itemize}
			\item[]
			$J_r$ is $\delta$ away from boundary layers and interior layers. 
			Without singularity, we could write the bound of $|u'|$ as $C_1$, 
			which 
			may depend on a minus power of $\delta$.
		\end{itemize}
		
		In the end, using the facts that $|f-\bar{f}|,|b-\bar b|\leq Ch$ 
		and 
		$|u|\leq C|f|$, the first-order convergence in $L^\infty$ norm is 
		proved.
	\end{proof}
	
	Applying approximations in Theorem \ref{thm:Linf}, we summarize the 
	algorithm as follows:
	\begin{algorithm}
		\label{alg_PGFEM}
		\rm \textbf{(PGFEM for turning point problems)}
		\begin{enumerate}
			\item Identify singular points $s_i$'s and their types of 
			singularities;
			\item Take a partition and add in singular points which are absent 
			in 
			the mesh;
			\item Approximate $p,b,f$ by piecewise constants or  
			piecewise linear functions based on distance from the 
			midpoint of an interval to singular points (see the proof of 
			Theorem \ref{thm:Linf});
			\item Solve dual problems analytically or numerically to		
			evaluate test functions;
			\item Generate stiffness matrix and right-hand-side term using 
			test functions;
			\item Solve a linear system to obtain the numerical solution 
			$u_h$ 
			(on 
			grid points).
		\end{enumerate}
	\end{algorithm}
	\begin{remark}
		\rm\label{remark_of_alg}
		This algorithm is a fitted operator method because we need no special 
		mesh 
		on the whole interval, and the solution is derived by selecting special 
		test functions. One needs to identify the location of singular 
		points in order to use information of the singularities and to solve 
		dual problems with enough precision to construct the linear system.
	\end{remark}
	\begin{theorem}
		\rm\label{thm:L2}
		\textbf{(First-order $L^2$ Uniform Convergence)} Providing the same 
		conditions as the previous theorem, PGFEM 
		is  
		uniformly convergent in $L^2$-norm and $\varepsilon$-norm, i.e., there 
		exists 
		a 
		constant $C$ independent of $h$ and $\varepsilon$, such that
		\begin{equation}
			\label{eq:thm:L2}
			||e||_{L^2}\leq||e||_{\varepsilon}\leq Ch.
		\end{equation}
	\end{theorem}
	\begin{proof}
		If we consider $L^2$-norm, we start by 
		distracting approximated equation \eqref{eq1'} from the original one
		\eqref{eq1}:
		$$-\varepsilon e''+\bar{p}e'+\bar{b}e=F(x)\equiv 
		-(p-\bar{p})u'-(b-\bar{b})u+(f-\bar{f}).$$
		We have proved that $|e|\leq C|F|\leq \tilde{C}h$. On the assumption 
		that $u,u_h\in C^1$, the 
		following energy 
		estimate holds:
		\begin{equation}
			\label{energy_estimate}
			\int_{x_L}^{x_R} \varepsilon 
			(e')^2+\Big(\bar{b}-\frac{1}{2}\bar{p}'\Big)e^2 
			dx\leq||F||_{L^2}||e||_{L^2}+
			\frac{1}{2}\sum_{i=1}^{N-1}e^2(x_i)[\bar{p}](x_i).
		\end{equation}
		For $h$ sufficiently small, using assumption in \eqref{eq:condition}, 
		it holds that
		\begin{displaymath}
			\bar{b}-\frac{1}{2}\bar{p}'\geq
			\frac{b_0+\gamma_0}{4}.
		\end{displaymath}
		Denoting
		$$\gamma_1 = \min(1,\frac{b_0+\gamma_0}{4}),$$
		we have the estimate in the energy norm from 
		\eqref{eq:thm:Linf} and \eqref{energy_estimate}:
		\begin{equation}
			\label{est2}
			||e||_\varepsilon^2\leq Ch^2,
		\end{equation}
		where the constant $C$ may depend on $p,b,f,\gamma^{-1},$	
		and the jump of $\bar{p}$ is at most $Ch$:
		$$|[\bar{p}](x_k)|\leq|\bar{p}(x_k^-)-p(x_k)|+|p(x_k)-\bar{p}(x_k^+)|\leq
		Ch.$$
	\end{proof}		
	\begin{proposition}
		\label{prop:stab}
		\rm \textbf{(Numerical stability)} The scheme satisfies discrete 
		maximum 
		principle, i.e., 
		the matrix induced by PGFEM is 
		tri-diagonally 
		dominated, which could be verified by taking 
		integrals\cite{farrell1988uniform}.
	\end{proposition}
	\section{Numerical implementation}\label{sec:4}
	In this section, we use three examples to validate the 
	efficiency and convergency of our algorithm.
	Different cases of singularities (a), (b), and (c) are included in these 
	examples. 
	Test functions could be computed with exact parabolic cylinder 
	functions or approximated by numerical solutions.
	PGFEM solutions with fine 
	girds (N=4096) using exact test functions are chosen to be reference 
	solutions in the 
	first two examples; the third one admits an exact solution. 
	We calculate errors by $||\cdot||_{L^\infty_h}$, 
	$||\cdot||_{L^2_h}$ and $||\cdot||_{\varepsilon,h}$ defined in 
	\eqref{eq:uinf}--\eqref{eq:upl2}. 
	\begin{example}
		\label{ex:cos}
		\rm Consider a turning point problem with a cusp-like interior layer 
		and an exponential-type boundary layer:
		\begin{displaymath}
			\left\{
			\begin{aligned}
				&-\varepsilon u''+\cos(2\pi x)u'+u=\frac{1}{1+x^2},\quad 
				0<x<1,\\
				&u(0)=1,\quad u(1)=2.\\
			\end{aligned}
			\right.
		\end{displaymath}
		
		There is an interior layer at $x=1/4$ and a boundary 
		layer at $x=1$, corresponding to cases (b) and (a). We set 
		$\{1/4,3/4,1\}$ to be singular points which need special care.
		The condition that $b-p'$ is bounded from below is not satisfied in 
		this case, while numerical experiments suggest that Algorithm 
		\ref{alg_PGFEM}
		works, for 
		exact test functions, when the repulsive turning point $x=3/4$ is 
		specially treated; for approximate dual solutions, we need to 
		neglect $x=3/4$ for the sake of stablity.
		In practice we take 
		$\delta$ in Theorem \ref{thm:Linf} as follows in all three examples:
		$$\delta = \min \{0.1,\quad \min_i|s_i-s_{i-1}|/3\}.$$
		
		The reference solution, together with numerical solutions using PGFEM 
		on the uniform mesh
		and an up-winding scheme on Shishkin mesh \cite{natesan2003parameter} 
		with 
		both $256$ grids are shown in 
		Figure \ref{fig_sol_cos}. Exact dual solutions are selected as test 
		functions. Compared to non-equidistant mesh of 
		Shishkin 
		type, 
		PGFEM needs 
		no special grids, and values on the uniform mesh points are highly 
		accurate. The solution using Shishkin mesh has a lower 
		resolution outside 
		the interior layer, which could be improved by mesh refinement. PGFEM 
		errors 
		in three different discrete norms versus grid 
		number $N$ are 
		drawn with a log-log plot in Figure \ref{fig_err_cos}, where one may 
		find a 
		nearly 
		second-order convergency. 
		\begin{figure}[ht]
			\subfloat[]{\label{fig:1A}
				\centering
				\includegraphics[width=7cm]{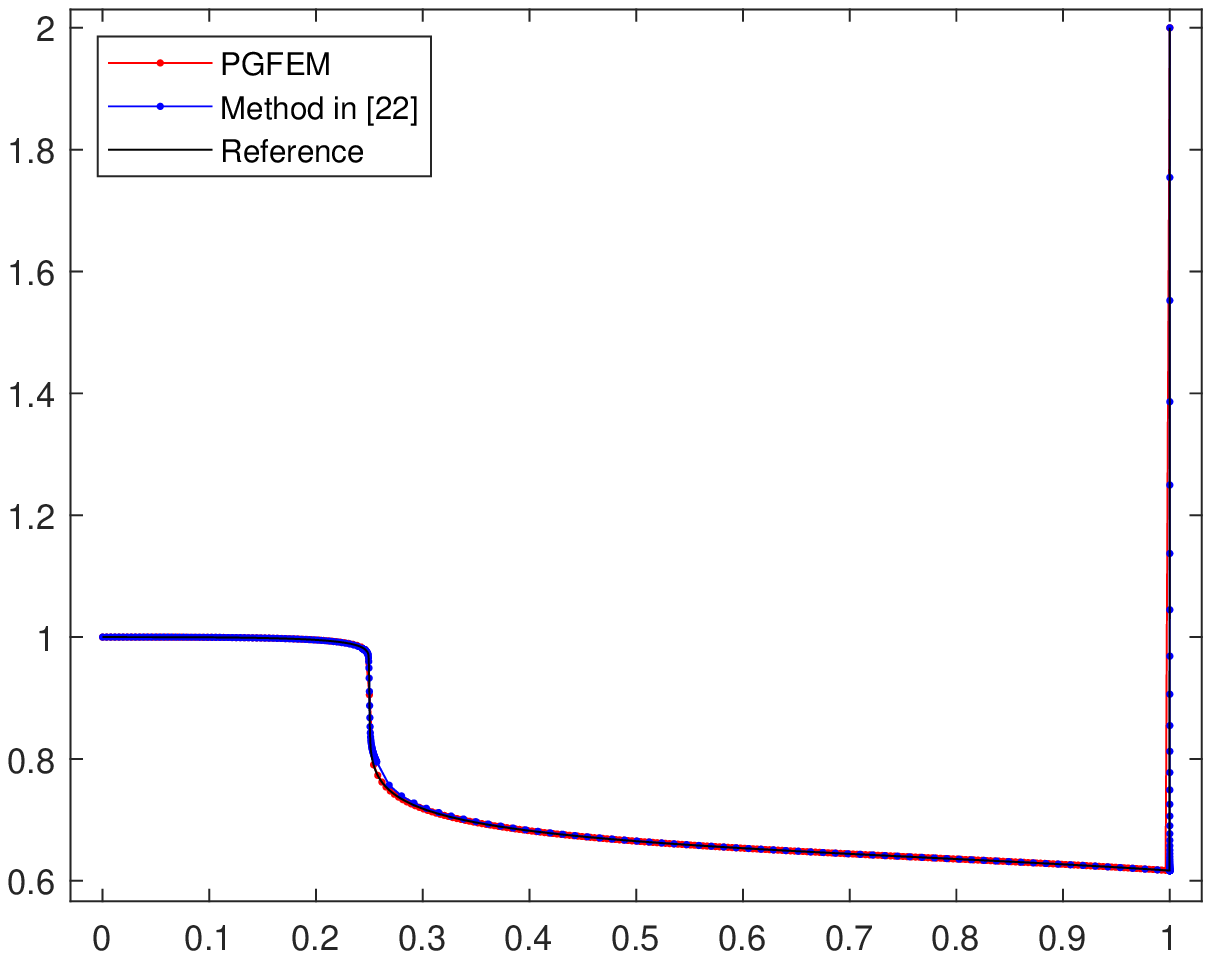}}
			\subfloat[]{\label{fig:1B}
				\centering
				\includegraphics[width=7cm]{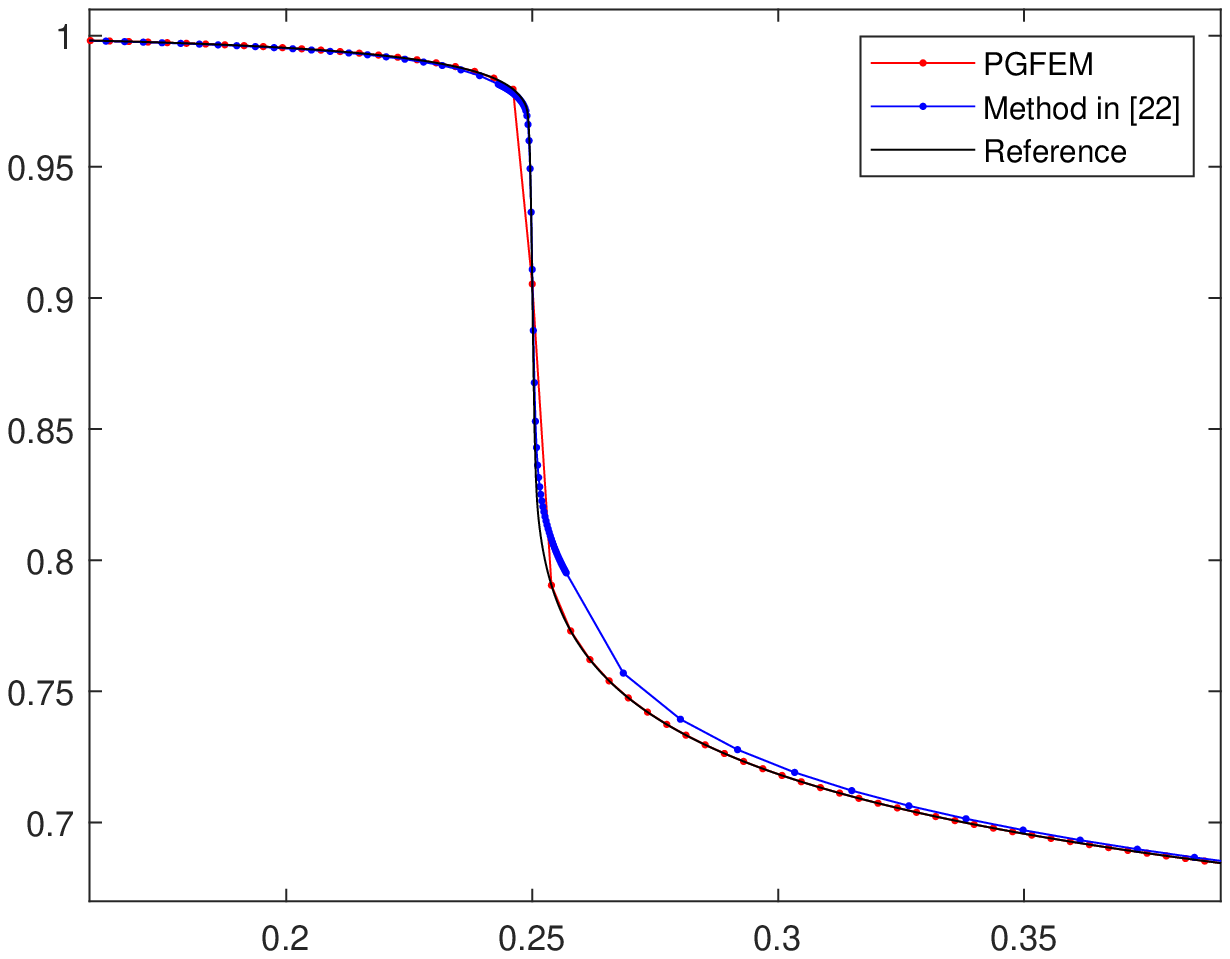}}
			\caption{\small Numerical and reference solutions for Example \ref{ex:cos}
				($\varepsilon=1\times10^{-6}$), where test functions in PGFEM 
				are 
				calculated 
				with 
				exact expressions. \textbf{(a)} PGFEM and the method 
				in \cite{natesan2003parameter} are employed with grid number 
				$N=256$, with the latter using an up-winding scheme on Shishkin 
				mesh. 
				The 
				reference solution is computed with PGFEM on grids 
				$N=4096$; \textbf{(b)} Horizontal Magnification near $x=1/4$.}
			\label{fig_sol_cos}
		\end{figure}
		\begin{figure}[htpb]
			\hspace{-14mm}\includegraphics[width=16cm]{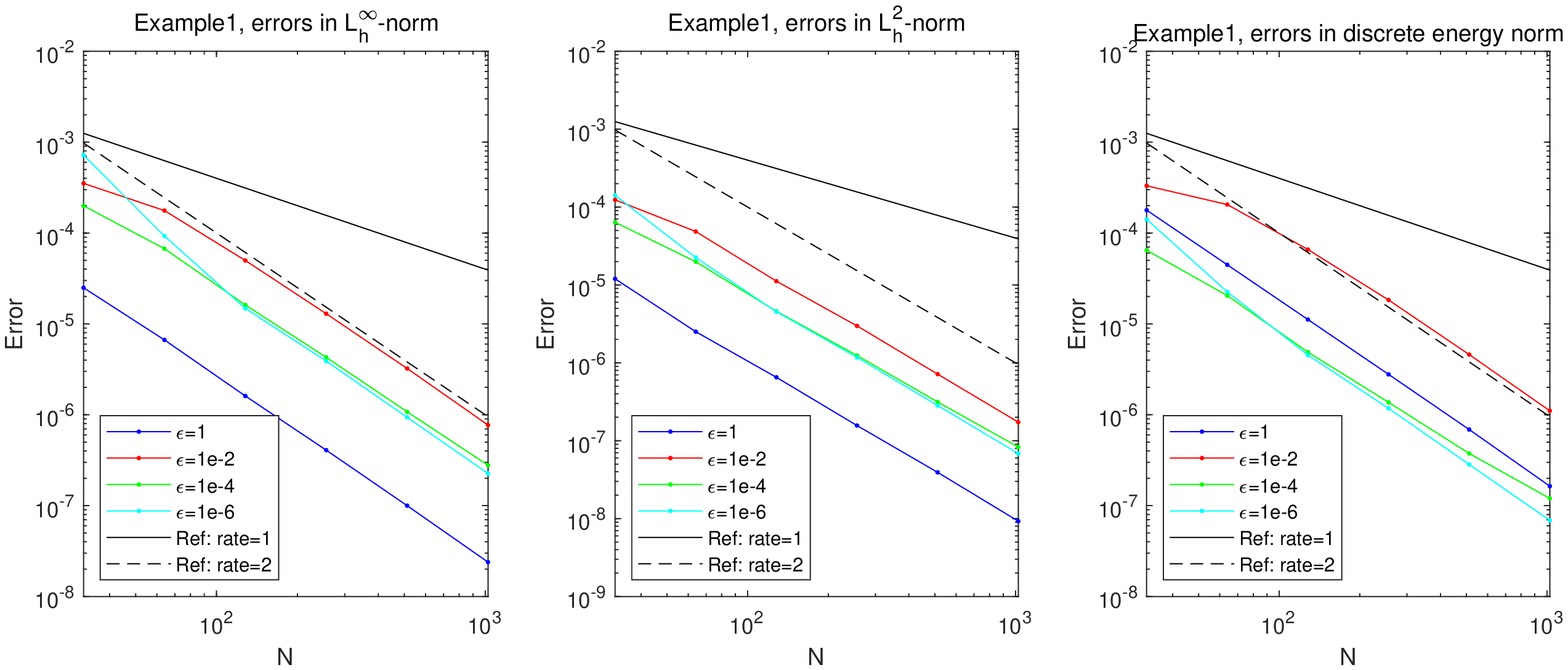}
			\caption{\small Log-log plot for PGFEM errors in Example \ref{ex:cos} versus 
				grid 
				number $N$, in $L^\infty_h,L^2_h$ and discrete energy norm. 
				The solid black line and black dashed 
				line have slopes -1 and -2, respectively.}
			\label{fig_err_cos}
		\end{figure}
	\end{example}
	\begin{example}
		\label{ex:single}
		\rm Consider a boundary turning point problem:
		\begin{displaymath}
			\left\{
			\begin{aligned}
				&-\varepsilon u''+(1-x^2)u'+3u=e^x,\quad -1<x<1,\\
				&u(-1)=1,\quad u(1)=2.\\
			\end{aligned}
			\right.
		\end{displaymath}
		
		At both endpoints the solution appears singular with $p'(-1)>0$ at 
		$x=-1$ and $p'(1)<0$ at $x=1$, corresponding to case (c) with a 
		positive slope and a negative one. These two boundary layers are weaker 
		than those 
		in Example
		\ref{ex:cos}, as the PGFEM solution and the reference solution are 
		drawn 
		in Figure
		\ref{fig_sol_single}. 
		Setting $\{-1,1\}$ as singular points and exact dual solutions as test 
		functions, $L^\infty_h$ errors and discrete 
		energy errors are shown in Tables \ref{table_Linferr_single} and 
		\ref{table_Eerr_single} accordingly, where a 
		first-order 
		uniform convergence could be verified.
		\begin{figure}[h]
			\subfloat[]{\label{fig:2A}
				\centering
				\includegraphics[width=7cm]{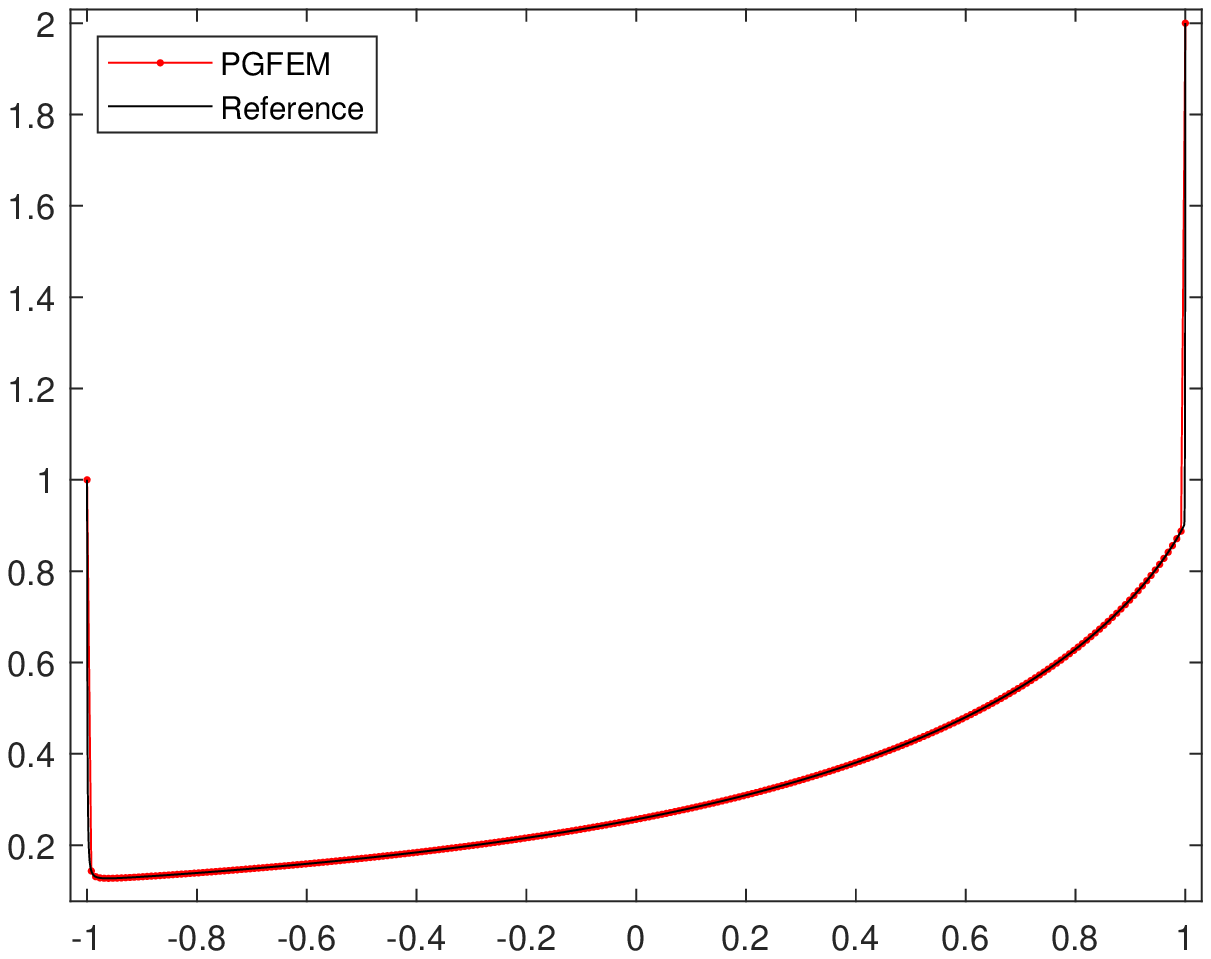}}
			\subfloat[]{\label{fig:2B}
				\centering
				\includegraphics[width=7cm]{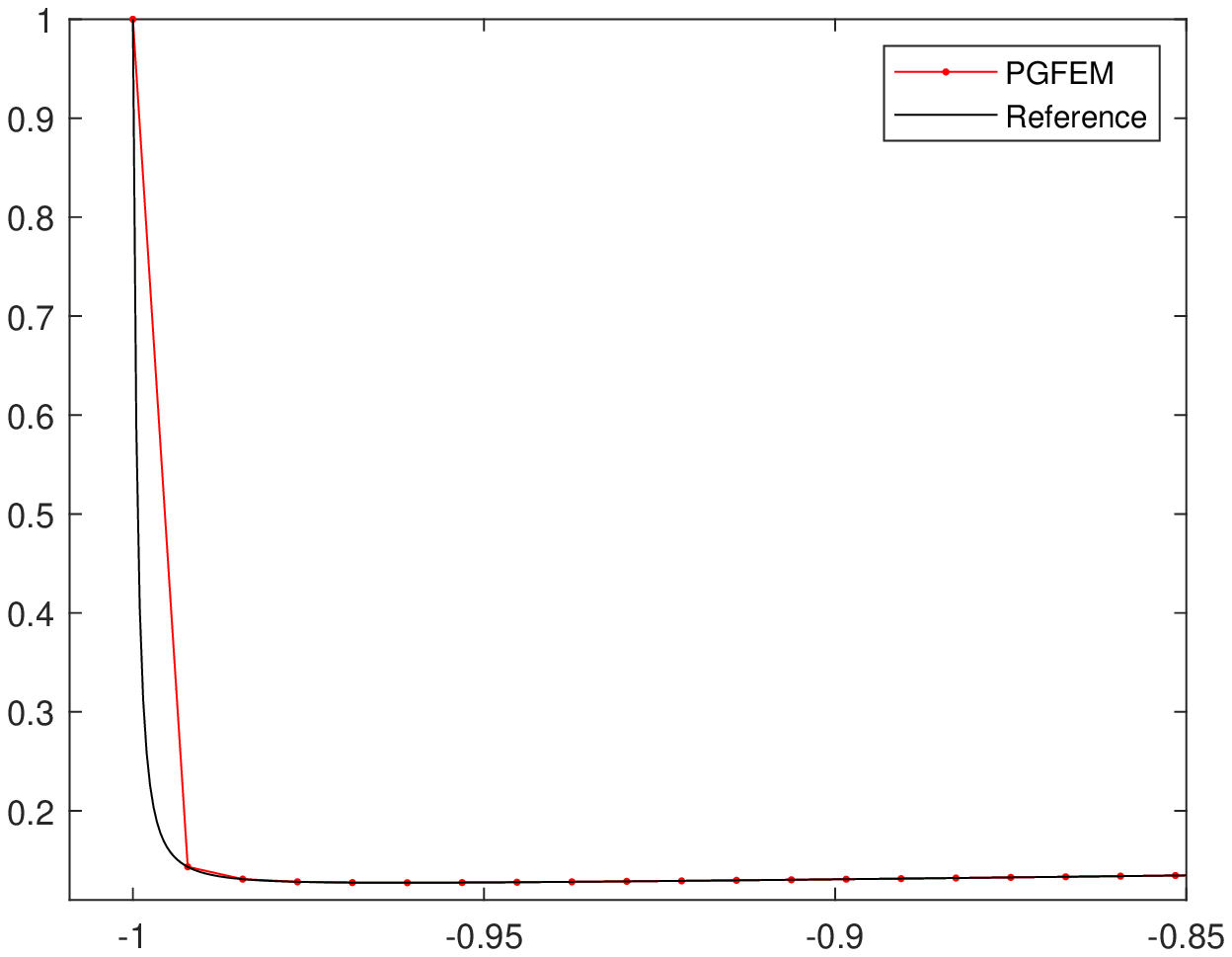}}
			\caption{\small Numerical and reference solutions for Example 
				\ref{ex:single} ($\varepsilon=1\times10^{-6}$). Test functions in PGFEM are 
				calculated 
				with 
				exact expressions. \textbf{(a)} PGFEM is implemented 
				with grids $N=256$, and the reference solution is calculated 
				using the 
				same algorithm with $N=4096$; 
				\textbf{(b)} Horizontal Magnification near $x=0$.}
			\label{fig_sol_single}
		\end{figure}
		\begin{table}[h]
			\caption{\small $L^\infty_h$ errors of PGFEM solutions for 
				Example \ref{ex:single}. Test functions are calculated with exact 
				expressions.}
				\centering
			\begin{tabular}{c|cc|cc|cc|cc}
				\hline
				$\varepsilon$ & \multicolumn{2}{c|}{1} & 
				\multicolumn{2}{c|}{1E-02} & \multicolumn{2}{c|}{1E-04} & 
				\multicolumn{2}{c}{1E-06} \\ \hline
				$N$           & $L_h^\infty$   & rate  & $L_h^\infty$     & 
				rate    & $L_h^\infty$     & rate    & $L_h^\infty$    & 
				rate    \\ \hline
				32            & 1.12E-04       &       & 2.78E-03         
				&         & 1.85E-03         &         & 1.85E-03        
				&         \\
				64            & 2.39E-05       & 2.43  & 1.46E-03         & 
				1.52    & 7.22E-04         & 2.10    & 7.16E-04        & 
				1.93    \\
				128           & 5.97E-06       & 2.00  & 3.72E-04         & 
				1.97    & 1.90E-04         & 1.93    & 1.83E-04        & 
				1.97    \\
				256           & 1.56E-06       & 1.98  & 7.86E-05         & 
				2.37    & 4.49E-05         & 2.26    & 8.66E-05        & 
				2.31    \\
				512           & 3.85E-07       & 2.02  & 1.94E-05         & 
				2.01    & 1.41E-05         & 1.89    & 3.73E-05        & 
				2.45    \\
				1024          & 9.07E-08       & 2.10  & 4.83E-06         & 
				2.04    & 3.33E-06         & 2.08    & 1.51E-05        & 
				2.22    \\ \hline
			\end{tabular}
			\label{table_Linferr_single}
		\end{table}
		
		\begin{table}[h]
			\caption{\small $||\cdot||_{\varepsilon,h} $ errors of 
				PGFEM solutions for Example \ref{ex:single}. Test functions are 
				calculated with exact expressions.}
			\label{table_Eerr_single}
			\centering
			\begin{tabular}{c|cc|cc|cc|cc}
				\hline
				$\varepsilon$ & \multicolumn{2}{c|}{1} & 
				\multicolumn{2}{c|}{1.E-02} & \multicolumn{2}{c|}{1.E-04} & 
				\multicolumn{2}{c}{1.E-06} \\ \hline
				$N$           & energy       & rate    & energy         & 
				rate       & energy         & rate       & energy         & 
				rate      \\ \hline
				32            & 3.93E-04     &         & 1.95E-03       
				&            & 6.38E-04       &            & 6.27E-04       
				&           \\
				64            & 9.56E-05     & 2.24    & 9.93E-04       & 
				1.52       & 2.08E-04       & 1.84       & 2.09E-04       & 
				1.77      \\
				128           & 2.39E-05     & 2.00    & 2.55E-04       & 
				1.97       & 5.31E-05       & 1.99       & 5.50E-05       & 
				2.01      \\
				256           & 6.02E-06     & 2.03    & 5.47E-05       & 
				2.37       & 1.51E-05       & 2.17       & 1.42E-05       & 
				2.18      \\
				512           & 1.49E-06     & 2.02    & 1.35E-05       & 
				2.01       & 4.46E-06       & 2.01       & 3.93E-06       & 
				2.15      \\
				1024          & 3.53E-07     & 2.08    & 3.36E-06       & 
				2.05       & 1.24E-06       & 2.03       & 1.14E-06       & 
				2.13      \\ \hline
			\end{tabular}
		\end{table}		
	\end{example}
	
	For Example \ref{ex:single}, if we compute test functions numerically, 
	using the 
	same reference solution, convergence is the same as above (see 
	Table \ref{table_Linferr_single_app}). Convergency also holds for multiple 
	turning point problems if we follow 
	the same procedure to compute test functions, although it is unclear
	what analytic expressions of the solutions to dual problems are.
	\begin{table}[h]
		\caption{\small $L^\infty_h$ errors of PGFEM solutions for 
			Example \ref{ex:single}. Test functions are approximated by numerical 
			solutions.}
		\label{table_Linferr_single_app}
		\centering
		\begin{tabular}{c|cc|cc|cc|cc}
			\hline
			$\varepsilon$ & \multicolumn{2}{c|}{1} & 
			\multicolumn{2}{c|}{1.E-02} & \multicolumn{2}{c|}{1.E-04} & 
			\multicolumn{2}{c}{1.E-06} \\ \hline
			$N$           & $L_h^\infty$   & rate  & $L_h^\infty$     & 
			rate     & $L_h^\infty$     & rate     & $L_h^\infty$     & rate    
			\\ \hline
			32            & 1.12E-04       &       & 2.78E-03         
			&          & 1.86E-03         &          & 1.86E-03         
			&         \\
			64            & 2.39E-05       & 2.43  & 1.46E-03         & 
			1.52     & 7.23E-04         & 2.11     & 7.21E-04         & 1.94    
			\\
			128           & 5.96E-06       & 2.00  & 3.72E-04         & 
			1.97     & 1.90E-04         & 1.93     & 1.84E-04         & 1.97    
			\\
			256           & 1.56E-06       & 1.98  & 7.86E-05         & 
			2.37     & 4.50E-05         & 2.26     & 8.66E-05         & 2.32    
			\\
			512           & 3.85E-07       & 2.02  & 1.94E-05         & 
			2.01     & 1.40E-05         & 1.89     & 3.68E-05         & 2.45    
			\\
			1024          & 9.07E-08       & 2.10  & 4.83E-06         & 
			2.04     & 3.33E-06         & 2.08     & 1.51E-05         & 2.29    
			\\ \hline
		\end{tabular}
	\end{table}	
	\begin{example}
		\rm\label{ex:multiple}
		Consider the following multiple boundary turning point problem 
		in\cite{vulanovic1993continuous}:
		\begin{displaymath}
			\left\{
			\begin{aligned}
				& -\varepsilon u''-x^3u'+u=f(x),\quad 0<x<1,\\
				& u(0)=2,\quad u(1)=e^{-1/\sqrt{\varepsilon}}+e,\\
			\end{aligned}
			\right.
		\end{displaymath}
		where $f(x)$ is determined by the 
		exact solution: $$u(x)=e^{-x/\sqrt{\varepsilon}}+e^x.$$
		
		There is a $\sqrt{\varepsilon}$-wide boundary layer at the turning 
		point $x=0$, as drawn in Figure \ref{fig_sol_multiple}. Results of 
		PGFEM 
		are 
		compared with 
		one in
		\cite{vulanovic1993continuous}, where we compute with two methods on 
		the same uniform mesh, and PGFEM obtains solutions with higher 
		precision. In this 
		case, analytic expressions of dual solutions are unknown. Thus we 
		utilize numerical solutions to dual problems as test functions.
		The $L^\infty_h$ convergence rate of PGFEM is almost two, as shown in 
		Table \ref{table_Linferr_multiple}.
		\begin{figure}[h]
			\subfloat[]{\label{fig:3A}
				\centering
				\includegraphics[width=7cm]{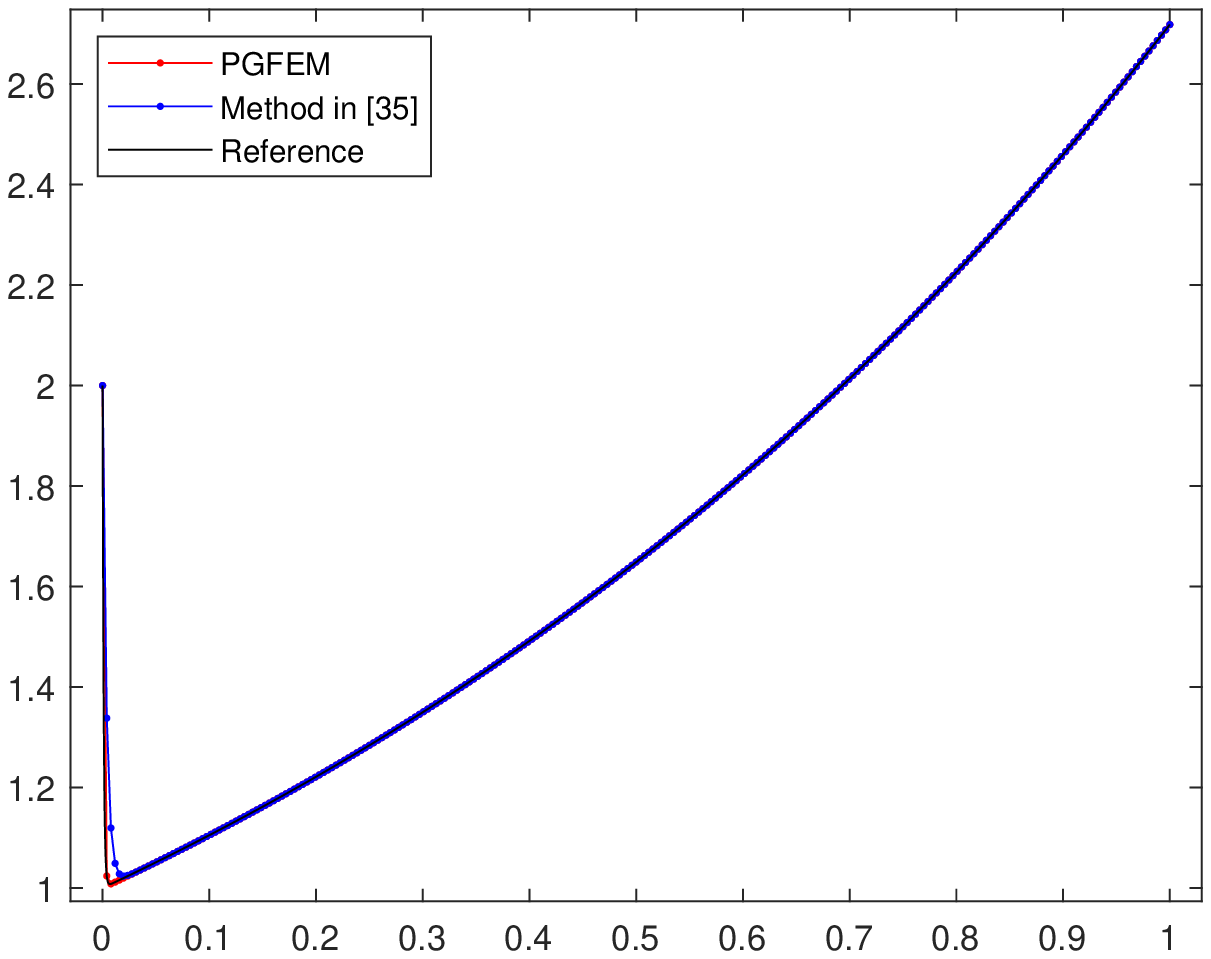}}
			\subfloat[]{\label{fig:3B}
				\centering
				\includegraphics[width=7cm]{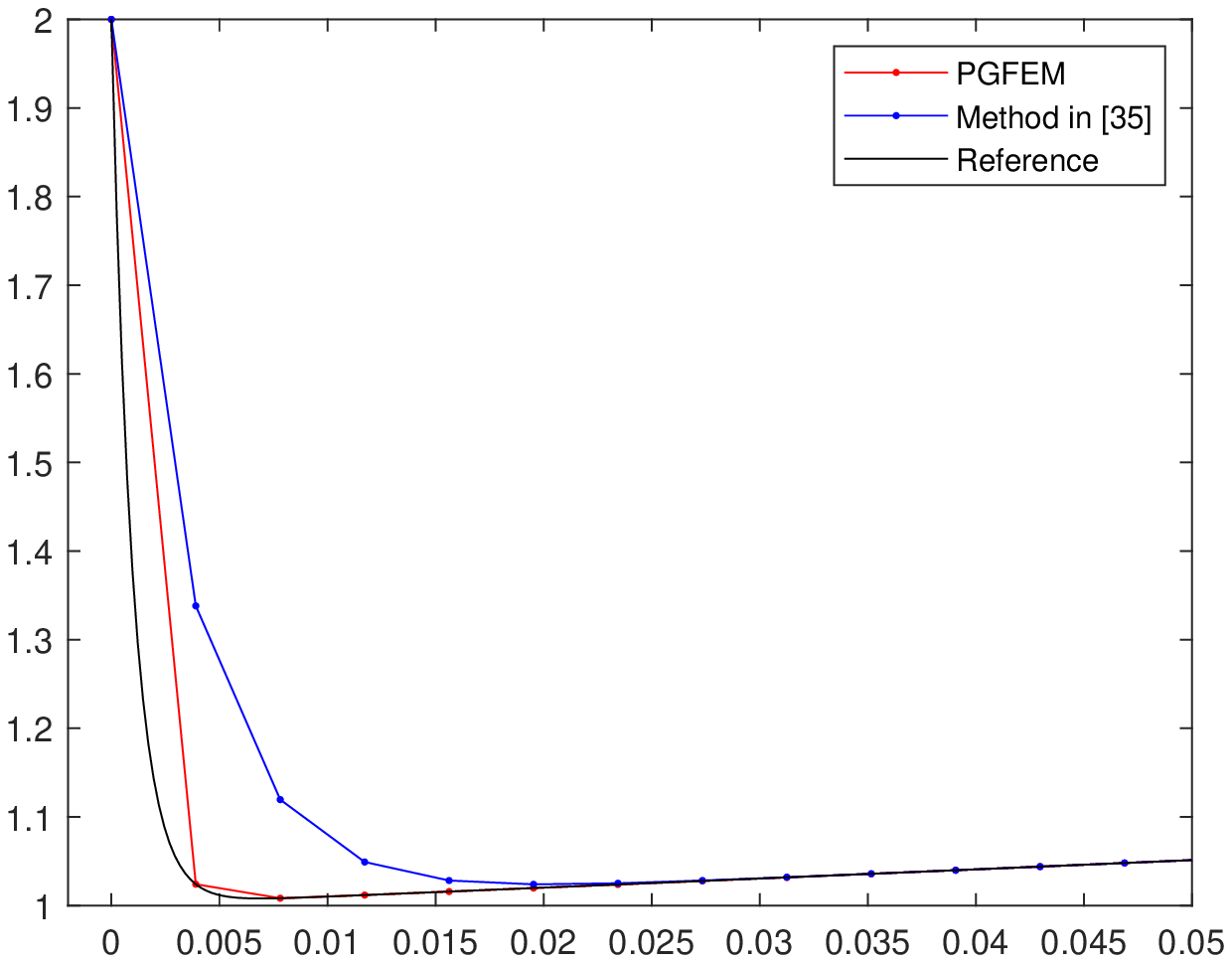}}
			\caption{\small Numerical and exact solutions for Example \ref{ex:multiple} 
				($\varepsilon=1\times10^{-6}$). Test functions in PGFEM are 
				approximated 
				by numerical solutions on grids $N=4096$.	\textbf{(a)} PGFEM 
				and the method 
				in \cite{vulanovic1993continuous} 
				are manipulated with uniform grids $N=256$, and the reference 
				solution is computed with PGFEM and
				$N=4096$; 
				\textbf{(b)} Horizontal Magnification near $x=0$.}
			\label{fig_sol_multiple}
		\end{figure}
		\begin{table}[h]
			\caption{\small $L^\infty_h$ errors of PGFEM solutions for 
				Example \ref{ex:multiple}. Test functions are approximated by numerical 
				solutions.}
			\label{table_Linferr_multiple}
			\centering
			\begin{tabular}{c|cc|cc|cc|cc}
				\hline
				$\varepsilon$ & \multicolumn{2}{c|}{1} & 
				\multicolumn{2}{c|}{1.E-02} & \multicolumn{2}{c|}{1.E-04} & 
				\multicolumn{2}{c}{1.E-06} \\ \hline
				$N$           & $L_h^\infty$   & rate  & $L_h^\infty$     & 
				rate     & $L_h^\infty$     & rate     & $L_h^\infty$     & 
				rate    \\ \hline
				32            & 1.84E-05       &       & 1.65E-04         
				&          & 3.71E-04         &          & 1.05E-03         
				&         \\
				64            & 4.61E-06       & 2.00  & 4.82E-05         & 
				1.77     & 5.89E-05         & 2.66     & 3.01E-04         & 
				1.81    \\
				128           & 1.15E-06       & 2.00  & 1.26E-05         & 
				1.93     & 8.85E-06         & 2.82     & 7.64E-05         & 
				1.98    \\
				256           & 2.88E-07       & 2.00  & 3.20E-06         & 
				1.98     & 2.22E-06         & 1.99     & 1.77E-05         & 
				2.11    \\
				512           & 7.21E-08       & 2.00  & 8.02E-07         & 
				2.00     & 5.50E-07         & 2.01     & 4.57E-06         & 
				2.31    \\
				1024          & 1.79E-08       & 2.01  & 2.01E-07         & 
				2.00     & 1.33E-07         & 2.04     & 1.17E-06         & 
				1.97    \\ \hline
			\end{tabular}
		\end{table}			
	\end{example}
	\section{Conclusion}
	In this paper, we develop a Petrov-Galerkin finite element method (PGFEM) 
	to solve a class of turning point problems 
	in one dimension. Priori estimates have been established for the single 
	boundary 
	turning point case. Numerical analysis shows that our scheme has
	first-order uniform convergency in several different norms. In numerical 
	examples, errors in different discrete norms validate the feasibility and 
	efficiency of the scheme. We emphasize that such an algorithm not only 
	could be implemented with evaluations of exact solutions to the dual 
	problems but also is considerable if test functions are approximated 
	numerically.

\begin{thebibliography}{10}
		\bibitem{abrahamson1977priori}
		Leif~R Abrahamson.
		\newblock A priori estimates for solutions of singular perturbations 
		with a
		turning point.
		\newblock {\em Stud. Appl. Math.}, 56:1 (1977), 51--69.
		
		\bibitem{berger1984priori}
		Alan~E Berger, Houde Han, and R~Bruce Kellogg.
		\newblock A priori estimates and analysis of a numerical method for a 
		turning
		point problem.
		\newblock {\em Math. Comput.}, 42:166 (1984), 465--492.
		
		\bibitem{broersen2014robust}
		Dirk Broersen and Rob Stevenson.
		\newblock A robust petrov--galerkin discretisation of 
		convection--diffusion
		equations.
		\newblock {\em Comput. Math. with Appl.}, 
		68:11 (2014), 1605--1618.
		
		\bibitem{chakraborty2020optimal}
		Ankit Chakraborty, Ajay Rangarajan, and Georg May.
		\newblock Optimal approximation spaces for discontinuous 
		petrov-galerkin finite
		element methods.
		\newblock {\em arXiv:2012.12751}, (2020).
		
		\bibitem{chen2008stability}
		Long Chen, Yonggang Wang, and Jinbiao Wu.
		\newblock Stability of a streamline diffusion finite element method for 
		turning
		point problems.
		\newblock {\em J. Comput. Appl. Math.},
		220:1-2 (2008), 712--724.
		
		\bibitem{de2011parameter}
		Carlo De~Falco and Eugene O'Riordan.
		\newblock A parameter robust Petrov--Galerkin scheme for
		advection--diffusion--reaction equations.
		\newblock {\em Numer. Algorithms}, 56:1 (2011), 107--127.
		
		\bibitem{de1979error}
		PPN De~Groen and PW Hemker.
		\newblock Error bounds for exponentially fitted galerkin methods 
		applied to
		stiff two-point boundary value problems, 
		\newblock in London a.o : Academic Press, GBR, 1979, pp. 
		217--249.
		
		\bibitem{el1978numerical}
		TM~El-Mistikawy and MJ~Werle.
		\newblock Numerical method for boundary layers with blowing-the 
		exponential box
		scheme.
		\newblock {\em AIAA Journal}, 16:7 (1978), 749--751.
		
		\bibitem{farrell1988sufficient}
		Paul~A Farrell.
		\newblock Sufficient conditions for the uniform convergence of a 
		difference
		scheme for a singularly perturbed turning point problem.
		\newblock {\em SIAM J. Numer. Anal.}, 25:3 (1988), 618--643. 		
		
		\bibitem{farrell1988uniform}
		Paul~A Farrell and Eugene~C Gartland~Jr.
		\newblock A uniform convergence result for a turning point problem,
		\newblock in {\em Proc. BAIL V Conference}, China, 1988, pp. 127--132.
		
		\bibitem{geng2014numerical}
		FZ~Geng, SP~Qian, and S~Li.
		\newblock A numerical method for singularly perturbed turning point 
		problems
		with an interior layer.
		\newblock {\em J. Comput. Appl. Math.}, 
		255 (2014), 97--105.
		
		\bibitem{guo1993uniformly}
		Wen Guo.
		\newblock {\em Uniformly convergent finite element methods for 
		singularly
			perturbed parabolic partial differential equations}.
		\newblock PhD thesis, University College Cork, 1993.
		
		\bibitem{han2010tailored}
		Houde Han and Zhongyi Huang.
		\newblock Tailored finite point method for steady-state 
		reaction-diffusion
		equations.
		\newblock {\em Commun. Math. Sci.}, 8:4 (2010), 887--899.
		
		\bibitem{han2013tailored}
		Houde Han and Zhongyi Huang.
		\newblock Tailored finite point method based on exponential bases for
		convection-diffusion-reaction equation.
		\newblock {\em Math. Comput.}, 82:281 (2013), 213--226.
		
		\bibitem{han2008tailored}
		Houde Han, Zhongyi Huang, and R~Bruce Kellogg.
		\newblock A tailored finite point method for a singular perturbation 
		problem on
		an unbounded domain.
		\newblock {\em J. Sci. Comput.}, 36:2 (2008), 243--261.
		
		\bibitem{hemker2000varepsilon}
		PW~Hemker, GI~Shishkin, and LP~Shishkina.
		\newblock $\varepsilon$-uniform schemes with high-order time-accuracy 
		for
		parabolic singular perturbation problems.
		\newblock {\em IMA J. Numer. Anal.}, 20:1 (2000), 99--121.
		
		\bibitem{huang2009tailored}
		Zhongyi Huang.
		\newblock Tailored finite point method for the interface problem.
		\newblock {\em Netw. Heterog. Media}, 4:1 (2009), 91--106.
		
		\bibitem{kellogg1978analysis}
		R~Bruce Kellogg and Alice Tsan.
		\newblock Analysis of some difference approximations for a singular
		perturbation problem without turning points.
		\newblock {\em Math. Comput.}, 32:144 (1978), 1025--1039.
		
		\bibitem{kumar2019parameter}
		Devendra Kumar.
		\newblock A parameter-uniform method for singularly perturbed turning 
		point
		problems exhibiting interior or twin boundary layers.
		\newblock {\em Int. J. Comput. Math.}, 
		96:5 (2019), 865--882.
		
		\bibitem{li2000convergence}
		Jichun Li.
		\newblock Convergence analysis of finite element methods for singularly
		perturbed problems.
		\newblock {\em Comput. Math. with Appl.}, 
		40:6-7 (2000), 735--745.
		
		\bibitem{munyakazi2019robust}
		Justin~B Munyakazi, Kailash~C Patidar, and Mbani~T Sayi.
		\newblock A robust fitted operator finite difference method for 
		singularly
		perturbed problems whose solution has an interior layer.
		\newblock {\em Math. Comput. Simul.}, 160 (2019), 155--167.
		
		\bibitem{natesan2003parameter}
		Srinivasan Natesan, J~Jayakumar, and J~Vigo-Aguiar.
		\newblock Parameter uniform numerical method for singularly perturbed 
		turning
		point problems exhibiting boundary layers.
		\newblock {\em J. Comput. Appl. Math.},
		158:1 (2003), 121--134.
		
		\bibitem{olver2010nist}
		Frank~WJ Olver, Daniel~W Lozier, Ronald~F Boisvert, and Charles~W Clark.
		\newblock {\em NIST handbook of mathematical functions hardback and 
		CD-ROM}.
		\newblock Cambridge university press, 2010.
		
		\bibitem{o2012singularly}
		Eugene O'Riordan and Jason Quinn.
		\newblock A singularly perturbed convection diffusion turning point 
		problem
		with an interior layer.
		\newblock {\em Comput. Methods Appl. Math.}, 
		12:2 (2012), 206--220.
		
		\bibitem{o1970boundary}
		RE~O’Malley, Jr.
		\newblock On boundary value problems for a singularly perturbed 
		differential
		equation with a turning point.
		\newblock {\em SIAM J. Math. Anal.}, 1:4 (1970), 479--490.
		
		\bibitem{o2011parameter}
		E~O’Riordan and J~Quinn.
		\newblock Parameter-uniform numerical methods for some linear and 
		nonlinear
		singularly perturbed convection diffusion boundary turning point 
		problems.
		\newblock {\em BIT Numer. Math.}, 51:2 (2011), 317--337.
		
		\bibitem{roos1990global}
		H-G Roos.
		\newblock Global uniformly convergent schemes for a singularly perturbed
		boundary-value problem using patched base spline-functions.
		\newblock {\em J. Comput. Appl. Math.}, 
		29:1 (1990), 69--77.
		
		\bibitem{stynes1986finite}
		Martin Stynes and Eugene O'Riordan.
		\newblock A finite element method for a singularly perturbed boundary 
		value
		problem.
		\newblock {\em Numer Math (Heidelb)}, 50:1 (1986), 1--15.
		
		\bibitem{stynes19871}
		Martin Stynes and Eugene O'Riordan.
		\newblock ${L}^{1}$ and ${L}^{\infty}$ uniform convergence of a 
		difference
		scheme for a semilinear singular perturbation problem.
		\newblock {\em Numer Math (Heidelb)}, 50:5 (1987), 519--531.
		
		\bibitem{sun1994finite}
		Guangfu Sun and Martin Stynes.
		\newblock Finite element methods on piecewise equidistant meshes for 
		interior
		turning point problems.
		\newblock {\em Numer. Algorithms}, 8:1 (1994), 111--129.
		
		\bibitem{sun1995finite}
		Guangfu Sun and Martin Stynes.
		\newblock Finite-element methods for singularly perturbed high-order 
		elliptic
		two-point boundary value problems. i: reaction-diffusion-type problems.
		\newblock {\em IMA J. Numer. Anal.}, 15:1 (1995), 117--139.
		
		\bibitem{temme2000numerical}
		Nico~M Temme.
		\newblock Numerical and asymptotic aspects of parabolic cylinder 
		functions.
		\newblock {\em J. Comput. Appl. Math.},
		121:1-2 (2000), 221--246.
		
		\bibitem{vulanovic1987non}
		Relja~Vulanovi{\'c}.
		\newblock Non-equidistant generalizations of the Gushchin-Shchennikov 
		scheme.
		\newblock {\em Z. Angew. Math. Mech.}, 67:12 (1987), 625--632.
		
		\bibitem{vulanovic1990numerical}
		Relja Vulanovi{\'c}.
		\newblock On numerical solution of a mildly nonlinear turning point 
		problem.
		\newblock {\em Esaim Math. Model. Numer. Anal.},
		24:6 (1990), 765--783.
		
		\bibitem{vulanovic1993continuous}
		Relja Vulanovi{\'c} and Paul~A Farrell.
		\newblock Continuous and numerical analysis of a multiple boundary 
		turning
		point problem.
		\newblock {\em SIAM J. Numer. Anal.}, 30:5 (1993), 1400--1418.
		
		\bibitem{yadav2021almost}
		Swati Yadav and Pratima Rai.
		\newblock An almost second order hybrid scheme for the numerical 
		solution of
		singularly perturbed parabolic turning point problem with interior 
		layer.
		\newblock {\em Math. Comput. Simul.}, 185 (2021), 733--753.
	\end{thebibliography}
	
\end{document}